\documentclass[a4paper,fleqn,leqno]{article}

\usepackage{latexsym,amssymb,amsthm,amsmath,amsfonts,mathtools,epsfig,color,epsf,graphicx,mathrsfs}
\usepackage[english]{babel}

\usepackage[utf8]{inputenc} 
\usepackage[T1]{fontenc}  
\usepackage{lmodern}  
\usepackage{tikz}
\usepackage{doi}
\usetikzlibrary{trees}
\usetikzlibrary{decorations.pathreplacing}
\usepackage{tabu}
\usepackage{array,multirow,multicol}
\usepackage{hyperref}       
\usepackage{url}            
\usepackage{booktabs}       
 \usepackage{nicefrac}       
\usepackage{microtype}      
\usepackage{lipsum}

\newcommand{\N}{\mathbb{N}}
\newcommand{\R}{\mathbb{R}}
\renewcommand{\P}{\mathbb{P}}
\newcommand{\Z}{\mathbb{Z}}
\newcommand{\E}{\mathbb{E}}

\renewcommand{\d}{{\rm d}}
\newcommand{\e}{{\rm e}}

\makeatletter
 \@addtoreset{equation}{section}
 \makeatother

 \newcounter{extralabel}[section]

 \newtheorem{ittheorem}{Theorem}
 \newtheorem{itlemma}{Lemma}
 \newtheorem{itproposition}{Proposition}
 \newtheorem{itdefinition}{Definition}
 \newtheorem{itcorollary}{Corollary}
 \newtheorem{itconjecture}{Conjecture}
 \newtheorem{itremark}{Remark}
 \theoremstyle{definition}
 \newtheorem{itassumption}{Assumption}

 \newenvironment{theorem}{\addtocounter{extralabel}{1}
 \begin{ittheorem}}{\end{ittheorem}}

 \newenvironment{lemma}{\addtocounter{extralabel}{1}
 \begin{itlemma}}{\end{itlemma}}

 \newenvironment{proposition}{\addtocounter{extralabel}{1}
 \begin{itproposition}}{\end{itproposition}}

 \newenvironment{definition}{\addtocounter{extralabel}{1}
 \begin{itdefinition}}{\end{itdefinition}}

 \newenvironment{corollary}{\addtocounter{extralabel}{1}
 \begin{itcorollary}}{\end{itcorollary}}

 \newenvironment{remark}{\addtocounter{extralabel}{1}
 \begin{itremark}}{\end{itremark}}

\newenvironment{assumption}{\addtocounter{extralabel}{1}
	\begin{itassumption}}{\end{itassumption}}
\title{Spatially Inhomogeneous Populations with Seed-banks\\
	\Large I. Duality, Existence and Clustering.
}

\begin{document}

\author{Frank den Hollander$^1$, Shubhamoy Nandan$^2$}

\date{\today}

\maketitle

\begin{abstract}
We consider a system of interacting Moran models with seed-banks. Individuals live in colonies and are subject to resampling and migration as long as they are \emph{active}. Each colony has a seed-bank into which individuals can retreat to become \emph{dormant}, suspending their resampling and migration until they become active again. The colonies are labelled by $\Z^d$, $d \geq 1$, playing the role of a \emph{geographic space}. The sizes of the active and the dormant population are \emph{finite} and depend on the \emph{location} of the colony. Migration is driven by a random walk transition kernel. Our goal is to study the equilibrium behaviour of the system as a function of the underlying model parameters.

In the present paper we show that, under mild condition on the sizes of the active population, the system is well-defined and has a dual. The dual consists of a system of \emph{interacting} coalescing random walks in an \emph{inhomogeneous} environment that switch between an active state and a dormant state. We analyse the dichotomy of \emph{coexistence} (= multi-type equilibria) versus \emph{clustering} (= mono-type equilibria), and show that clustering occurs if and only if two random walks in the dual starting from arbitrary states eventually coalesce with probability one. The presence of the seed-bank \emph{enhances genetic diversity}. In the dual this is reflected by the presence of time lapses during which the random walks are dormant and do not move.



\medskip\noindent 
\emph{Acknowledgements:} 
The research in this paper was supported by the Netherlands Organisation for Scientific Research (NWO) through grant TOP1.17.019. The authors thank Simone Floreani, Cristian Gardin\`a and Frank Redig for various discussions on duality.

\end{abstract}

\bigskip

\footnoterule
\noindent
\hspace*{0.3cm} {\footnotesize $^{1)}$ 
Mathematisch Instituut, Universiteit Leiden, Niels Bohrweg 1, 2333 CA  Leiden, NL\\
denholla@math.leidenuniv.nl}\\
\hspace*{0.3cm} {\footnotesize $^{2)}$ 
Mathematisch Instituut, Universiteit Leiden, Niels Bohrweg 1, 2333 CA  Leiden, NL\\
s.nandan@math.leidenuniv.nl}

%%%%%%%%%%%%%%%%%%%%%%%%%%%%%%%%%%%%%%

\newpage

\tableofcontents

\newpage

%%%%%%%%%%%%%%%%%% SECTION 1 %%%%%%%%%%%%%%%%%%%%%%%%

\section{Background, motivation and outline}

Dormancy is an evolutionary trait observed in plants, bacteria and other microbial populations, where an organism enters a reversible state of low metabolic activity as a response to adverse environmental conditions. The dormant state of an organism in a population is characterised by interruption of basic reproduction and phenotypic development during periods of environmental stress \cite{LJ11,LS}. The dormant organisms reside in what is called a \emph{seed-bank} of the population. After a varying and possibly large number of generations, dormant organisms can be resuscitated under more favourable conditions and reprise reproduction after becoming \emph{active} by leaving the seed-bank. This strategy is known to have important implications for the genetic diversity and overall fitness of the underlying population \cite{LJ11,LHBB}, since the seed-bank of a population often acts as a \emph{buffer} against evolutionary forces such as genetic drift, selection and environmental variability. The importance of dormancy has led to several attempts to model seed-banks from a mathematical perspective (\cite{BCEK15,BCKW16}; see also \cite{BK} for a broad overview).

In \cite{BCEK15} and \cite{BCKW16}, the Fisher-Wright model with \emph{seed-bank} was introduced and analysed. In the Fisher-Wright model with seed-bank, individuals live in a colony, are subject to \emph{resampling} where they adopt each other's type, and move in and out of the seed-bank where they suspend resampling. The seed-bank acts as a repository for the genetic information of the population. Individuals that reside inside the seed-bank are called \emph{dormant}, those that reside outside are called \emph{active}. Both the long-time behaviour and the genealogy of the population were analysed for the continuum model obtained by letting the size of the colony tend to infinity, called the Fisher-Wright diffusion with seed-bank. 

In \cite{GHO1}, \cite{GHO2}, \cite{GHO3} the continuum model was extended to a \emph{spatial} setting in which individuals live in multiple colonies, labelled by a countable Abelian group playing the role of a \emph{geographic space}. In the spatial model with seed-banks, each colony is endowed with its own seed-bank and individuals are allowed to \emph{migrate} between colonies. The goal was to understand the change in behaviour compared to the spatial model without seed-bank. 

Most papers on seed-banks deal with the large-colony-size limit, for which the evolution is described by a system of coupled SDE's. In \cite{HP}, a multi-colony Fisher-Wright model with seed-banks was introduced where the colony sizes are finite. However, this model is restricted to \emph{homogeneous} population sizes and a finite geographic space. The present paper introduces an individual-based spatial model with seed-banks in continuous time where the sizes of the underlying populations are \emph{finite} and \emph{vary} across colonies. The latter make the model more interesting from a biological perspective, but raise extra technical challenges. The key tool that we use to tackle these challenges is \emph{stochastic duality}\cite{GKRV,CGGR}. The spatial model introduced in this paper fits in the realm of interacting particle systems, which often admit additional structures such as duality\cite{L,RS}. In particular, our spatial model can be viewed as a \emph{hybrid} of the well-known Voter Model and the \emph{generalized Symmetric Exclusion Process}, $2j$-SEP, $j\in\N/2$ \cite{CGRS,GKRV,L99}. Both the Voter Model and the $2j$-SEP enjoy the stochastic duality property, and our system inherits this as well: it is dual to a system consisting of \emph{coalescing} random walks with \emph{repulsive} interactions. The resulting dual process shares striking resemblances with the dual processes of the Voter Model and $2j$-SEP, because the original process is a modified hybrid of them. It has been recognised in the literature \cite{VGO,LHBB,LJ11,BCEK15,BCKW16} that qualitatively different behaviour may occur when the exit time of a typical individual from the seed-bank can become large. In the present paper we are able to model this phenomenon as well, due to the \emph{inhomogeneity} in the seed-bank sizes. Our main goals are the following: 
\begin{itemize}
	\item[(1)]
	Introduce a model with seed-banks whose size is \emph{finite} and depends on the geographic location of the colony. Prove \emph{existence} and \emph{uniqueness} of the process via well-posedness of an associated martingale problem and duality with a system of interacting coalescing random walks.   
	\item[(2)] 
	Identify a criterion for {\em coexistence} (= convergence towards multi-type equilibria) and \emph{clustering} (= convergence towards mono-type equilibria). Show that there is a one-parameter family of equilibria controlled by the density of types.
	\item[(3)] 
	Identify the \emph{domain of attraction} of the equilibria. 
	\item [(4)] 
	Identify the \emph{parameter regime} under which the criterion for clustering is met. In case of clustering, find out how fast the mono-type clusters grow in space-time. In case of coexistence, establish mixing properties of the equilibria.
\end{itemize}
In the present paper we settle (1) and (2). In~\cite{HN01} we will address (3) and (4). We focus on the situation where the individuals can be of \emph{two types}. The extension to infinitely many types, called the Fleming-Viot measure-valued diffusion, only requires standard adaptations and will not be considered here.  

The paper is organised as follows. In Section~\ref{s.dichotomy} we give a quick definition of the model and state our main theorems about the well-posedness, the duality and the clustering criterion. In Section~\ref{s.basics} we give a more detailed definition of the model, prove that the martingale problem associated with its generator is well-posed, establish duality with an interacting seed-bank coalescent, demonstrate that the system exhibits a dichotomy between clustering and coexistence, and formulate a necessary and sufficient condition for clustering to prevail in terms of the dual, called the \emph{clustering criterion}. Sections~\ref{s.duality}--\ref{s.criterion} are devoted to the proof of our main theorems.

%%%%%%%%%% SECTION 2 %%%%%%%%%%%%%%%%%%%%%%%%%%

\section{Main theorems}
\label{s.dichotomy}

In Section~\ref{ss.quick} we give a quick definition of the system. In Section~\ref{ss.wpd} we argue that, under mild conditions on the sizes of the active population, the system is well-defined and has a dual that consists of finitely many interacting coalescing random walks.
%%%

\subsection{Quick definition of the multi-colony system} 
\label{ss.quick}

Individuals live in colonies labelled by $\Z^d$, $d \geq 1$, which plays the role of a \emph{geographic space}. (In what follows, the geographic space can be any countable Abelian group.) Each colony has an \emph{active} population and a \emph{dormant} population. Each individual carries one of two \emph{types}: $\heartsuit$ and $\spadesuit$. Individuals are subject to: 
\begin{itemize}
	\item[(1)]
	Active individuals in any colony \emph{resample} with active individuals in \emph{any} colony.
	\item[(2)]
	Active individuals in any colony \emph{exchange} with dormant individuals in the \emph{same} colony.
\end{itemize} 
For (1) we assume that each active individual at colony $i$ at rate $a(i,j)$ uniformly draws an active individual at colony $j$ and \emph{adopts its type}. For (2) we assume that each active individual at colony $i$ at rate $\lambda$ uniformly draws a dormant individual at colony $i$ and the two individuals \emph{trade places while keeping their type} (i.e., the active individual becomes dormant and the dormant individual becomes active). Note that dormant individuals do \emph{not} resample.

At each colony $i$ we register the pair $(X_i(t),Y_i(t))$, representing the number of active, respectively, dormant individuals of type $\heartsuit$ at time $t$ at colony $i$. We write $(N_i,M_i)$ to denote the \emph{size} of the active, respectively, dormant population at colony $i$. The resulting Markov process is denoted by
\begin{equation}
	\label{process}
	(Z(t))_{t \geq 0}, \qquad Z(t) = ((X_i(t),Y_i(t))_{i \in \Z^d},
\end{equation}
and lives on the state space
\begin{equation}
	\mathcal{X} := \prod_{i\in\Z^d} [N_i]\times[M_i],
\end{equation}
where $[n] = \{0,1,\ldots,n\}$, $n \in \N$. In Section~\ref{ss.multi} we will show that, under mild assumptions on the model parameters, the Markov process in \eqref{process} is well defined and has a \emph{dual} $(Z^*(t))_{t \geq 0}$. The latter consists of finite collections of particles that perform \emph{interacting coalescing} random walks, with rates that are controlled by the model parameters. 

Let $\mathcal{P}$ be the set of probability distributions on $\mathcal{X}$ defined by
\begin{equation} 
	\mathcal{P} := \big\lbrace\mathcal{P}_\theta\colon\,\theta \in [0,1]\big\rbrace,
	\qquad \mathcal{P}_\theta := \theta \bigotimes_{i\in\Z^d}\delta_{(0,0)} + (1-\theta) \bigotimes_{i\in\Z^d}\delta_{(N_i,M_i)}.
\end{equation}
We say that \eqref{process} exhibits \emph{clustering} if the distribution of $Z(t)$ converges to a limiting distribution $\mu\in\mathcal{P}$ as $t\to\infty$. Otherwise, we say that it exhibits \emph{coexistence}. In Section~\ref{ss.multi} we will show that clustering is equivalent to \emph{coalescence} occurring eventually with probability $1$ in the dual consisting of \emph{two} particles. This will be the main route to the dichotomy.

For simplicity we let the exchange rate $\lambda \in (0,\infty)$ be the same for every colony, and let the \emph{migration kernel} be translation invariant and irreducible.

\begin{assumption}{\bf [Homogeneous migration]}
	\label{ass:kernel}
	The migration kernel $a(\cdot,\cdot)$ satisfies:
	\begin{itemize}
		\label{assumpt1}
		\item $a(\cdot,\cdot)$ is irreducible in $\Z^d$.
		\item 
		$a(i,j) = a(0,j-i)$ for all $i,j\in\Z^d $.
		\item  
		$c:=\displaystyle\sum_{i\in\Z^d} a(0,i) < \infty$ and $a(0,0)=\frac{1}{2}$. \hfill$\Box$
	\end{itemize}
\end{assumption}

\noindent
The former of the last two assumptions ensures that the way genetic information moves between colonies is homogeneous in space, while the latter ensures that the total rate of resampling is finite and that resampling is possible also at the same colony. Since it is crucial for our analysis that the population sizes remain constant, we view migration as a change of types without the individuals actually moving themselves. In this way, genetic information moves between colonies while the individuals themselves stay put.

We write
\begin{equation}
	K_i=\frac{N_i}{M_i}, \qquad i\in\Z^d, 
\end{equation}
to denote the \emph{ratio} of the size of the active and the dormant population in colony $i$.

%%%

\subsection{Well-posedness and duality}
\label{ss.wpd}

\begin{theorem}{\bf [Well-posedness and duality]}
	\label{T.WPD}
	Suppose that Assumption~{\rm \ref{assumpt1}} is in force. Then the Markov process $(Z(t))_{t\geq 0}$ in \eqref{process} has a factorial moment dual $(Z^*(t))_{t\geq 0}$ living in the state space $\mathcal{X}^*\subset \mathcal{X}$ consisting of all  configurations with finite mass, and the martingale problem associated with \eqref{process} is well-posed under either of the two following conditions:
	\begin{itemize}
		\item[(a)]
		$\lim_{\|i\| \to \infty} \|i\|^{-1} \log N_i = 0$ and $\sum_{i\in\Z^d} \e^{\delta \|i\|} a(0,i)<\infty$ for some $\delta>0$,
		\item[(b)]
		$\sup_{i\in\Z^d\backslash\{0\}} \|i\|^{-\gamma} N_i < \infty$ and $\sum_{i\in\Z^d} \|i\|^{d+\gamma+\delta} a(0,i)<\infty$ for $\gamma>0$ and some $\delta>0$.
	\end{itemize}
\end{theorem}

\noindent
Theorem~\ref{T.WPD} provides us with two sufficient conditions under which the system is well-defined and has a tractable dual. It shows a \emph{trade-off}: the more we restrict the tails of the migration kernel, the less we need to restrict the sizes of the active population. The sizes of the dormant population play no role because all the events (resampling, migration and exchange) in our model are initiated by active individuals and dormant individuals do not feel the spatial extent of the geographic space. Theorem~\ref{T.Mul_Dual}, Corollary~\ref{C.Mul_Dual} and Theorem~\ref{T.Mul_Well} in Section~\ref{ss.multi} contain the fine details.

%%%

\subsection{Equilibrium: coexistence versus clustering}
\label{ss.equi}

\begin{theorem}{\bf [Equilibrium]}
	\label{T.EQU}
	If the initial distribution of the system is such that each active and each dormant individual adopts a type with the same probability independently of other individuals, then the system admits a one-parameter family of equilibria.
	\begin{itemize}
		\item 
		The family of equilibria is parameterised by the probability to have one of the two types. 
		\item  
		The system converges to a mono-type equilibrium if and only if two random walks in the dual starting from arbitrary states eventually coalesce with probability one.
	\end{itemize}
\end{theorem}

\noindent
Theorem~\ref{T.EQU} tells us that the system converges to an equilibrium when it is started from a specific class of initial distributions, namely, products of binomials. It also provides a \emph{criterion} in terms of the dual that determines whether the equilibrium is mono-type or multi-type. Theorem~\ref{T.Equil}, Corollary~\ref{C.Equil} and Theorem~\ref{T.Dichotomy} in Section~\ref{ss.multi} contain the fine details.

%%%%%%%%%%%%%%%%%% SECTION 3 %%%%%%%%%%%%%%%%%%%%%%%%%

\section{Basic theorems: duality, well-posedness and clustering criterion}
\label{s.basics}

In Section~\ref{ss.single} we define and analyse the single-colony model. In Section~\ref{ss.multi} we do the same for the multi-colony model. Our focus is on well-posedness, duality and convergence to equilibrium.    

%%%

\subsection{Single-colony model}
\label{ss.single}

%%%

\subsubsection{Definition: resampling and exchange}

Consider two populations, called \emph{active} and \emph{dormant}, consisting of $N$ and $M$ haploid individuals, respectively. Individuals in the population carry one of two genetic types: $\heartsuit$ and $\spadesuit$. Dormant individuals reside inside the \emph{seed-bank}, active individuals reside outside. The dynamics of the single-colony Moran model with seed-bank is as follows:
\begin{itemize}
	\item[--]
	Each individual in the active population carries a \emph{resampling clock} that rings at rate $1$. When the clock rings, the individual randomly chooses an active individual and \emph{adopts} its type. 
	\item[--]
	Each individual in the active population also carries an \emph{exchange clock} that rings at rate $\lambda$. When the clock rings, the individual randomly chooses a dormant individual and exchanges state, i.e., becomes dormant and forces the chosen dormant individual to become active. During the exchange the two individuals \emph{retain} their type.   
\end{itemize}
Since the sizes of the two populations remain constant, we only need two variables to describe the dynamics of the population, namely, the number of a type-$\heartsuit$ individuals in both populations (see Table~\ref{tab:single}). 

%%%%%%%%%%%%%%%%%%%%%%%%%%%%%%%%%%%%%%%%%%%%%%%
\begin{table}[htbp]
	\renewcommand{\arraystretch}{1.2}
	\begin{tabular}{|c|c|c|c|}
		\hline
		\multicolumn{1}{|c|}{Initial state} & \multicolumn{1}{c|}{Event} & \multicolumn{1}{c|}{Final state} 
		& \multicolumn{1}{c|}{Transition rate} \\ \hline
		\multirow{4}{*}{$(x,y)$}      
		& \multirow{2}{*}{Resampling}     
		&  $(x-1,y)$                
		&  $\nicefrac{x(N-x)}{N}$ \\ \cline{3-4} 
		&                       
		&   $(x+1,y)$            
		&   $\nicefrac{x(N-x)}{N}$ \\ \cline{2-4} 
		& \multirow{2}{*}{Exchange}     
		&   $(x-1,y+1)$           
		&   $\nicefrac{\lambda x(M-y)}{M}$ \\ \cline{3-4} 
		&                       
		&   $(x+1,y-1)$             
		&   $\nicefrac{\lambda (N-x)y}{M}$ \\ \hline
	\end{tabular}
	\caption{Scheme of transitions in the single-colony model.}
	\label{tab:single}
\end{table}
%%%%%%%%%%%%%%%%%%%%%%%%%%%%%%%%%%%%%

Let $x$ and $y$ denote the number of individuals of type $\heartsuit$ in the active and the dormant population, respectively. After a resampling event, $(x,y)$ can change to $(x-1,y)$ or $(x+1,y)$, while after an exchange event $(x,y)$ can change to $(x-1,y+1)$ or $(x+1,y-1)$. Both changes in the resampling event occur at rate $x\frac{N-x}{N}$. In the exchange event, however, to see $(x,y)$ change to $(x-1,y+1)$, an exchange clock of a type-$\heartsuit$ individual in the active population has to ring (which happens at rate $\lambda x$), and that individual has to choose a type-$\spadesuit$ individual in the dormant population (which happens with probability $\frac{M-y}{M}$). Hence the total rate at which $(x,y)$ changes to $(x-1,y+1)$ is $\lambda x \frac{M-y}{M}$. By the same argument, the total rate at which $(x,y)$ changes to $(x+1,y-1)$ is $\lambda (N-x)\frac{y}{M}$.

For convenience we multiply the rate of resampling by a factor $\frac{1}{2}$, in order to make it compatible with the Fisher-Wright model. Thus, the generator $G$ of the process is given by
\begin{equation} 
	\label{eqn1}
	G = G_{\text{Mor}} + G_{\text{Exc}},
\end{equation}
where 
\begin{equation}
	\label{eqn2}
	(G_{\text{Mor}}f)(x,y) = \frac{x(N-x)}{2N}\left[f(x-1,y)+f(x+1,y)-2f(x,y)\right]
\end{equation}
describes the Moran resampling of active individuals at rate $\frac{1}{2}$ and
\begin{equation}\
	\label{eqn3}
	(G_{\text{Exc}}f)(x,y) = \frac{\lambda}{M}x(M-y)\left[f(x-1,y+1)-f(x,y)\right] 
	+ \frac{\lambda}{M}y(N-x)\left[f(x+1,y-1)-f(x,y)\right]
\end{equation}
describes the exchange between active and dormant individuals at rate $\lambda$. From here onwards, we denote the Markov process associated with the generator $G$ by 
\begin{equation}
	Z = (Z(t))_{t \geq 0}, \qquad Z(t) = (X(t),Y(t)),
\end{equation} 
where $X(t)$ and $Y(t)$ are the number of type-$\heartsuit$ active and dormant individuals at time $t$, respectively. The process $Z$ has state space $[N]\times [M]$, where $[N]=\{0,1,\ldots,N\}$ and $[M]=\{0,1,\ldots,M\}$. Note that $Z$ is well-defined because it is a continuous-time Markov chain with finitely many states.

%%%

\subsubsection{Duality and equilibrium}

The classical Moran model is known to be dual to the block-counting process of the Kingman coalescent. In this section we show that the single-colony Moran model with seed-bank also has a coalescent dual.

\begin{definition}{\bf [Block-counting process]}
	\label{def1}
	{\rm The \emph{block-counting process} of the interacting seed-bank coalescent (defined in Definition~\ref{def:isc} below) is the continuous-time Markov chain 
		\begin{equation}
			Z^* = (Z^*(t))_{t \geq 0}, \qquad Z^*(t) = (n_t,m_t),
		\end{equation} 
		taking values in the state space $[N] \times [M]$ with transition rates
		\begin{equation}
			\label{eqn4}
			(n,m)\mapsto 
			\begin{array}{ll}
				(n-1,m+1) &\text{ at rate } \lambda n \left(1-\tfrac{m}{M}\right),\\[0.2cm]
				(n+1,m-1) & \text{ at rate } \lambda K m \left(1-\tfrac{n}{N}\right),\\[0.2cm]
				(n-1,m)    &\text{ at rate } \tfrac{1}{N}\binom{n}{2}\mathbf{1}_{\{n\geq 2\}},
			\end{array}
		\end{equation}
		where $K=\frac{N}{M}$ is the \emph{ratio} of the sizes of the active and the dormant population.} \hfill $\Box$
\end{definition}

\noindent
The first two transitions in \eqref{eqn4} correspond to exchange, the third transition to resampling. Later in this section we describe the associated \emph{interacting seed-bank coalescent} process, which gives the genealogy of $Z$. 

The following result gives the duality between $Z$ and $Z^*$.

\begin{theorem}{\bf [Duality]}
	\label{T.Sing_Dual}
	The process $Z$ is dual to the process $Z^*$ via the duality relation
	\begin{equation}
		\label{eqn5}
		\E_{(X,Y)}\left[\frac{\binom{X(t)}{n}}{\binom{N}{n}}\frac{\binom{Y(t)}{m}}{\binom{M}{m}}
		\mathbf{1}_{\{n\leq X(t),m\leq Y(t)\}}\right]
		=\E^{(n,m)}\left[\frac{\binom{X}{n(t)}}{\binom{N}{n(t)}}\frac{\binom{Y}{m(t)}}{\binom{M}{m(t)}}
		\mathbf{1}_{\{n(t)\leq X,m(t)\leq Y\}}\right],\quad t\geq 0,
	\end{equation}
	where $\E$ stands for generic expectation. On the left the expectation is taken over $Z$ with initial state $Z(0)=(X,Y)\in[N]\times[M]$, on the right the expectation is taken over $Z^*$ with initial state $Z^*(0)=(n,m)\in[N]\times[M]$.
\end{theorem}

\noindent
Note that the duality relation fixes the factorial moments and thereby the mixed moments of the random vector $(X(t),Y(t))$. This enables us to determine the equilibrium distribution of $Z$. 

Although the above duality is new in the literature on seed-banks, the notion of factorial duality is not uncommon in mathematical models involving finite and fixed population sizes \cite{EG,G78}. Similar types of dualities are often found for other models too (e.g. self-duality of independent random walks, exclusion and inclusion processes, etc. \cite{GKRV}). Remarkably, in the special case where $N=M=2j$ for some $j\in\N/2,$ Giardin{\`a} et al. (2009) \cite[Section 3.2]{GKRV} identified the same duality relation as in \eqref{eqn5} as a self-duality for the generalized $2j$-SEP on two-sites. This is not surprising given the fact that the exchange rates between active and dormant individuals defined in Table~\ref{tab:single} are precisely the rates (up to rescaling) for the $2j$-SEP on two sites. We refer the reader to Section~\ref{ss.repr} to gain further insights into this.

\begin{proposition}{\bf [Convergence of moments]}
	\label{prop1}
	For any $(X,Y),(n,m)\in[N]\times[M]$ with $(n,m)\neq (0,0)$,
	\begin{equation}
		\lim_{t\to\infty}\E_{(X,Y)}\left[X(t)^nY(t)^m\right]=N^nM^m\frac{X+Y}{N+M}.
	\end{equation}
\end{proposition}

\noindent
Since the vector $(X(t),Y(t))$ takes values in $[N]\times[M]$, which has $(N+1)(M+1)$ points, the above proposition determines the limiting distribution of $(X(t),Y(t))$.

\begin{corollary}{\bf [Equilibrium]}
	\label{T.Sin_Eq}
	Suppose that $Z$ starts from initial state $(X,Y)\in[N]\times[M]$. Then $(X(t),Y(t))$ converges in law as $t\to\infty$ to a random vector $(X_\infty,Y_\infty)$ whose distribution is given by
	\begin{equation}
		\label{eqn6}
		\mathcal{L}_{(X,Y)}(X_\infty,Y_\infty)=\frac{X+Y}{N+M}\,\delta_{(N,M)}+\left(1-\frac{X+Y}{N+M}\right)\delta_{(0,0)}.
	\end{equation}
\end{corollary}

\noindent
Note that the equilibrium behaviour of $Z$ is the same as for the classical Moran model without seed-bank. The fixation probability of type $\heartsuit$ is $\tfrac{X+Y}{N+M}$, which is nothing but the initial frequency of type-$\heartsuit$ individuals in the \emph{entire population}. Even though the presence of the seed-bank delays the time of fixation, because its size is finite size it has no significant effect on the overall qualitative behaviour of the process. We will see in Section~\ref{ss.multi} that the situation is different in the multi-colony model.

%%%

\subsubsection{Interacting seed-bank coalescent}
\label{sub1}

In our model, the \emph{genealogy} of a sample taken from the finite population of $N+M$ individuals is governed by a partition-valued coalescent process similarly as for the genealogy of the classical Moran model. However, due the presence of the seed-bank, blocks of a partition are marked as $A$ (active) and $D$ (dormant). Unlike in the genealogy of the classical Moran model, the blocks \emph{interact} with each other. This interaction is present because of the restriction to \emph{finite size} of the active and the dormant population. For this reason, we name the block process an \emph{interacting seed-bank coalescent}. For convenience, we will use the word lineage to refer to a block in a partition.

Let $\mathcal{P}_k$ be the set of partitions of $\{1,2,\ldots,k\}$. For $\xi\in\mathcal{P}_k$, denote the number of lineages in $\xi$ by $|\xi|$. Furthermore, for $j,k,l\in\N$, define 
\begin{equation}
	\mathcal{M}_{j,k,l} = \Big\{\vec{u}\in\{A,D\}^{j}\colon\,\text{ the numbers of $A$ and $D$ in } \vec{u}
	\text{ are at most $k$ and $l$, respectively}\Big\}.
\end{equation}
The state space of the process is $\mathcal{P}_{N,M}=\{(\xi,\vec{u})\,\colon\,\xi\in\mathcal{P}_{N+M},\vec{u}\in\mathcal{M}_{|\xi|,N,M}\}$. Note that $\mathcal{P}_{N,M}$ contains only those marked partitions of $\{1,2,\ldots,N+M\}$ that have at most $N$ active lineages and $M$ dormant lineages. This is because we can only sample at most $N$ active and $M$ dormant individuals from the population. 

Before we give the formal definition, let us adopt some notation. For $\pi,\pi^\prime\in\mathcal{P}_{N,M}$, we say that $\pi\succ\pi^\prime$ if $\pi^\prime$ can be obtained from $\pi$ by merging two active lineages. Similarly, we say that $\pi\Join\pi^\prime$ if $\pi^\prime$ can be obtained from $\pi$ by altering the state of a single lineage ($A\to D$ or $D\to A$). We write $|\pi|_A$ and $|\pi|_D$ to denote the number of active and dormant lineages present in $\pi$, respectively.

\begin{definition}{\bf [Interacting seed-bank coalescent]}
	\label{def:isc}
	{\rm The \emph{interacting seed-bank coalescent} is the continuous-time Markov chain with state space $\mathcal{P}_{M,N}$ characterised by the following transition rates:
		\begin{equation}
			\label{eqn7}
			\begin{aligned}
				&\pi\mapsto\pi^\prime \text{ at rate }\\
				&\begin{array}{ll}
					\frac{1}{N} &\text{ if } \pi \succ \pi^\prime,\\[0.2cm]
					\lambda\big(1-\frac{|\pi|_D}{M}\big) &\text{ if } \pi \Join \pi^\prime 
					\text{ by change of state of one lineage in $\pi$ from $A$ to $D$},\\[0.2cm]
					\lambda K\big(1-\frac{|\pi|_A}{N}\big) &\text{ if } \pi\Join\pi^\prime 
					\text{ by change of state of one lineage in $\pi$ from $D$ to $A$}.
				\end{array}
			\end{aligned}
		\end{equation}
	} \hfill$\Box$
\end{definition}

\noindent
The factor $1-\frac{|\pi|_D}{M}$ in the transition rate of a single active lineage when $\pi$ becomes dormant reflects the fact that, as the seed-bank gets full, it becomes more difficult for an active lineage to enter the seed-bank. Similarly, as the number of active lineages decreases due to the coalescence, it becomes easier for a dormant lineage to leave the seed-bank and become active. This also tells us that there is a \emph{repulsive interaction} between the lineages of the same state ($A$ or $D$). Due to this interaction, it is tricky to study the coalescent. As $N,M$ get large, the interaction becomes weak. As $N,M\to\infty$, after proper space-time scaling, the coalescent converges weakly to a limit coalescent where the interaction is no longer present. In fact, it can be shown that when both the time and the parameters are scaled properly, the coalescent converges weakly as $N,M\to\infty$ to the \emph{seed-bank coalescent} described in \cite{BCKW16}.

We can also describe the coalescent in terms of an interacting particle system with the help of a graphical representation (see Figure~\ref{fig:1}). The interacting particle system consists of two reservoirs, called \emph{active} reservoir and \emph{dormant} reservoir, having $N$ and $M$ labeled sites, respectively, each of which can be occupied by at most one particle. The particles in the active and dormant reservoir are called \emph{active} and \emph{dormant} particles, respectively. The active particles can coalesce with each other, in the sense that if an active particle occupies a labeled site where an active particle is present already, then the two particles are glued together to form a single particle at that site. Active particles can become dormant by moving to an empty site in the dormant reservoir, while dormant particles can become active by moving to an empty site in the active reservoir. The transition rates are as follows:
\begin{itemize}
	\item 
	An active particle tries to coalesce with another active particle at rate $\frac{1}{2}$ by choosing uniformly at random a labeled site in the active reservoir. If the chosen site is empty, then it ignores the transition, otherwise it coalesces with the active particle present at the new site.
	\item 
	An active particle becomes dormant at rate $\lambda$ by moving to a random labeled site in the dormant reservoir when the chosen site is empty, otherwise it remains in the active reservoir.
\end{itemize}

%%%%%%%%%%%%%%%%%%%%%%%%%%%%%%%%%%%%%%%%%%
\begin{figure}[htbp]
	\begin{center}
		\begin{tikzpicture}[yscale=0.5]
			\draw [fill=blue!20!] (-0.5,-0.5) rectangle (5.5,7.5);
			\draw [fill=red!20!] (10,-0.5) rectangle (12,7.5);
			\foreach \x in {0,1,2,3,4,5,10.5,11.5}
			{
				\draw [fill=black, thick] (\x,-0.5) circle [radius=0.05];	
			}
			\draw [thick,black,->] (-1,-0.5) to (-1,7.5) node[left,xshift=-0.1cm,yshift=-3cm] {t};
			\node[align=left, below] at (2.5,-0.5)
			{ Active reservoir $(N=6)$};
			\node[below] at (11,-0.5)
			{ Dormant reservoir $(M=2)$};
			\draw[thick,blue] (0,0) to (0,4.3);
			\draw[thick,blue] (0,4.3) to (1,4.3);
			\draw[thick,blue] (1,0) to (1,7);
			\draw[thick,blue] (1,3) to (2,3);
			\draw[thick,blue] (2,0) to (2,3);
			\draw[thick,blue] (3,0) to (3,7);
			\draw[thick,blue] (4,0) to (4,2.5);
			\draw[thick,black,dashed,->] (4,2.5) to (10.5,2.5) node[black,above,xshift=-2.9cm] 
			{ rate $\lambda(\tfrac{M-m}{M})=\tfrac{\lambda}{2}$};
			\draw[thick,red] (10.5,2.5) to (10.5,6.3);
			\draw[thick,black,dashed,<-] (5,6.3) to (10.5,6.3) node[above,xshift=-2.9cm] 
			{ rate $\lambda K (\tfrac{N-n}{N})=2\lambda$};
			\draw[thick,blue] (5,6.3) to (5,7);
			\draw[thick,blue] (5,0) to (5,1.3);
			\draw[thick,black,dashed,->] (5,1.3) to (11.5,1.3) node[black,below,xshift=-3.5cm] 
			{ rate $\lambda(\tfrac{M-m}{M})=\lambda$};
			\draw[thick,red] (11.5,1.3) to (11.5,7);
			\draw[thick,black,dashed,->] (3,5) to (10,5);
			\node[above] at (7.7,5.1) {Dormant reservoir is full.};
			\node[thick,red] at (7.5,5) {$X$};
		\end{tikzpicture}
	\end{center}
	\caption{\small Scheme of transitions for an interacting particle system with an active reservoir of size $N=6$ and a dormant reservoir of size $M=2$, so that $K=\tfrac{N}{M}=\tfrac{6}{2}=3$. The effective rate for each of $n$ active particles to become dormant is $\lambda\tfrac{M-m}{M}$ when the dormant reservoir has $m$ particles. Similarly, the effective rate for each of $m$ dormant particles to become active is $\lambda K \tfrac{N-n}{N}$ when the active reservoir has $n$ particles.}
	\label{fig:1}
\end{figure}
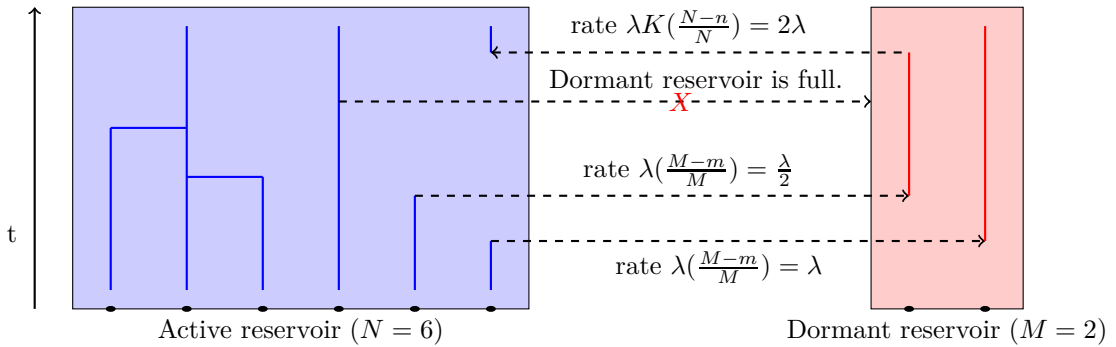
%%%%%%%%%%%%%%%%%%%%%%%%%%%%%%%%%%%%%%%%%%%%

\begin{itemize}
	\item 
	A dormant particle becomes active at rate $\lambda K$ by moving to a random labeled site in the active reservoir when the chosen site is empty, otherwise it remains in the dormant reservoir.
\end{itemize}

\noindent
Clearly, the particles \emph{interact with each other} due to the finite capacity of the two reservoirs. If $N,M\to\infty$, then the probability to choose an empty site in a reservoir tends to 1, and so the system converges (after proper scaling) to an interacting particle system where the particles move independently between the two reservoirs.

Note that if we define $n_t=$ number of active particles at time $t$ and $m_t=$ number of dormant particles at time $t$, then $Z^*=(n_t,m_t)_{t\geq 0}$ is the block-counting process defined in Definition \ref{def1}. Also, if we remove the labels of the sites in the two reservoirs and represents the particle configuration by an element of $\mathcal{P}_{N,M}$, then we obtain the \emph{interacting seed-bank coalescent} described in Definition~\ref{def:isc}. Even though it is natural to describe the genealogical process via a partition-valued stochastic process, we will stick with the interacting particle system description of the dual, since this will be more convenient for the multi-colony model.

%%%

\subsection{Multi-colony model}
\label{ss.multi}

In this section we consider multiple colonies, each with their own seed-bank. Each colony has an \emph{active} population and a \emph{dormant} population. We take $\Z^d$ as the underlying \emph{geographic space} where the colonies are located (any countable Abelian group will do). With each colony $i\in \Z^d$ we associate a variable $(X_i,Y_i)$, with $X_i$ and $Y_i$ the number of type-$\heartsuit$ active and dormant individuals, respectively, at colony $i$. Let $(N_i,M_i)$ denote the size of the active and the dormant population at colony $i$. In each colony active individuals are subject to resampling and migration, and to exchange with dormant individuals that are in the same colony. Dormant individuals are not subject to resampling and migration.

Since it is crucial for our duality to keep the population sizes constant, we consider migration of types without the individuals actually moving themselves. To be precise, by a migration from colony $j$ to colony $i$ we mean that an active individual from colony $i$ randomly chooses an active individual from colony $j$ and adopts its type. In this way, the \emph{genetic information} moves from colony $j$ to colony $i$, while the individuals themselves stay put.

%%%

\subsubsection{Definition: resampling, exchange and migration}

We assume that each active individual at colony $i$ resamples from colony $j$ at rate $a(i,j)$, adopting the type of a uniformly chosen active individual at colony $j$. Here, the \emph{migration kernel} $a(\cdot,\cdot)$ is assumed to satisfy Assumption~\ref{ass:kernel}. After a migration to colony $i$, the only variable that is affected is $X_i$, the number of type-$\heartsuit$ active individuals at colony $i$. The final state can be either $X_i-1$ or $X_i+1$ depending on whether a type-$\heartsuit$ active individual from colony $i$ chooses a type-$\spadesuit$ active individual from another colony or a type-$\spadesuit$ active individual from colony $i$ chooses a type-$\heartsuit$ active individual from another colony. The rate at which $X_i$ changes to $X_i-1$ due to a migration from colony $j$ is 
$$
a(i,j)X_i\tfrac{N_j-X_j}{N_j},
$$
while the rate at which $X_i$ changes to $X_i+1$ due to a migration from colony $j$ is 
$$
a(i,j)(N_i-X_i)\tfrac{X_j}{N_j}.
$$
Note that for $i=j$ the migration rate is 
$$
a(i,i)X_i\tfrac{N_i-X_i}{N_i}=\tfrac{X_i(N_i-X_i)}{2N_i},
$$ 
which is the same as the effective birth and death rate in the single-colony Moran model. Thus, the resampling within each colony is already taken care of via the migration. 

It remains to define the associated exchange mechanism between the active and the dormant individuals in a colony. The exchange mechanism is the same as in the single-colony model, i.e., in each colony each active individual at rate $\lambda$ performs an exchange with a dormant individual chosen uniformly from the seed-bank of that colony. For simplicity, we take the exchange rate $\lambda$ to be the same in each colony.

The state space $\mathcal{X}$ of the process is 
\begin{equation}
	\mathcal{X} := \prod_{i\in\Z^d}\{0,1,\ldots,N_i\}\times\{0,1,\ldots,M_i\}=\prod_{i\in\Z^d} [N_i]\times[M_i].
\end{equation}
A configuration $\eta\in\mathcal{X}$ is denoted by $\eta = (X_i,Y_i)_{i\in\Z^d}$, with $X_i\in[N_i]$ and $Y_i\in[M_i]$.

%%%%%%%%%%%%%%%%%%%%%%%%%%%%%%%%%%%%%%%%%%%%%%%
\begin{table}[htbp]
	\renewcommand{\arraystretch}{1.3}
	\begin{tabular}{|c|c|c|c|}
		\hline
		\multicolumn{1}{|c|}{Initial state} & \multicolumn{1}{c|}{Event} & \multicolumn{1}{c|}{Final state} & \multicolumn{1}{c|}{Transition rate} \\ \hline
		\multirow{4}{*}{$(X_i,Y_i)_{i\in\Z^d}$}      
		& \multirow{2}{8em}{Migration from colony $j$ to $i$}     
		& $(\cdots,(X_i-1,Y_i),\cdots)$                
		& $\nicefrac{a(i,j)X_i(N_j-X_j)}{N_j}$ \\ \cline{3-4} 
		&                       
		& $(\cdots,(X_i+1,Y_i),\cdots)$            
		& $\nicefrac{a(i,j)(N_i-X_i)X_j}{N_j}$ \\ \cline{2-4} 
		& \multirow{2}{*}{Exchange at colony $i$}     
		& $(\cdots,(X_i-1,Y_i+1),\cdots)$           
		& $\nicefrac{\lambda X_i(M_i-Y_i)}{M_i}$ \\ \cline{3-4} 
		&                       
		& $(\cdots,(X_i+1,Y_i-1),\cdots)$             
		& $\nicefrac{\lambda (N_i-X_i)Y_i}{M_i}$ \\ \hline
	\end{tabular}
	\caption{Scheme of transitions in the multi-colony model.}
	\label{tab:multi}
\end{table}
%%%%%%%%%%%%%%%%%%%%%%%%%%%%%%%%%%%%%

\noindent For each $i\in\Z^d$, let $\vec{\delta}_{i,A}=(X_n,Y_n)_{n\in\Z^d}$ and $\vec{\delta}_{i,D}=(\hat{X}_n,\hat{Y}_n)_{n\in\Z^d}$ be the configurations defined as
\begin{equation}
	\label{kron_del}
	(X_n,Y_n):= (\mathbf{1}_{\{n=i\}},0),\quad
	(\hat{X}_n,\hat{Y}_n):=(0,\mathbf{1}_{\{n=i\}})
	\qquad \forall \,n\in\Z^d.
\end{equation}
For two configurations $\eta = (\bar{X}_i,\bar{Y}_i)_{i\in\Z^d}$ and $\xi = (\hat{X}_i,\hat{Y}_i)_{i\in\Z^d}$, we define $\eta\pm \xi := (X_i,Y_i)_{i\in\Z^d}\in\mathcal{X}$ by setting, for each $i\in\Z^d$,
\begin{equation}
	\label{eq:addition-subtraction}
	\begin{aligned}
		X_i &= (\bar{X_i}\,\pm\,\hat{X_i})\mathbf{1}_{\{0\leq \bar{X}_i\pm\hat{X}_i\leq N_i\}}
		+N_i\,\mathbf{1}_{\{ \bar{X}_i\pm\hat{X}_i> N_i\}},\\
		Y_i &= (\bar{Y_i}\,\pm\,\hat{Y}_i)\mathbf{1}_{\{0\leq \bar{Y}_i\pm\hat{Y}_i\leq M_i\}}
		+M_i\,\mathbf{1}_{\{ \bar{Y}_i\pm\hat{Y}_i> M_i\}}.
	\end{aligned}
\end{equation}
Throughout the remainder of this paper, we adopt the convention given in \eqref{eq:addition-subtraction} for addition and subtraction of configurations in $\mathcal{X}$.

The generator $L$ for the process, acting on functions in 
\begin{equation}
	\label{eqnD}
	\mathcal{D} = \big\{f\in C(\mathcal{X})\colon\, f \text{ depends on finitely many coordinates}\big\},
\end{equation}
is given by
\begin{equation}
	\label{eqn8}
	L=L_\text{Mig}+L_\text{Res}+L_\text{Exc},
\end{equation}
where
\begin{equation}
	\label{eqn9}
	\begin{aligned}
		(L_{\text{Mig}}f)(\eta)
		&= \sum_{i\in\Z^d}\sum_{\overset{j\in\Z^d,}{j\neq i}}
		\frac{a(i,j)}{N_j}\Big\{X_i(N_j-X_j)[f(\eta-\vec{\delta}_{i,A})-f(\eta)]\\[-0.5cm]
		&\qquad\qquad\qquad\qquad\qquad +X_j(N_i-X_i)[f(\eta+\vec{\delta}_{i,A})-f(\eta)]\Big\}
	\end{aligned}
\end{equation}
describes the resampling of active individuals in \emph{different} colonies (= migration),
\begin{equation}
	\label{eqn10}
	(L_\text{Res}f)(\eta)
	= \sum_{i\in\Z^d}\frac{X_i(N_i-X_i)}{2N_i}[f(\eta-\vec{\delta}_{i,A})+f(\eta+\vec{\delta}_{i,A})-2f(\eta)]
\end{equation}
describes the resampling of active individuals in the \emph{same} colony, and
\begin{equation}
	\label{eqn11}
	\begin{aligned}
		&(L_\text{Exc}f)(\eta)\\
		&=\sum_{i\in\Z^d}\frac{\lambda}{M_i}\Big\{X_i(M_i-Y_i)[f(\eta-\vec{\delta}_{i,A}
		+\vec{\delta}_{i,D})-f(\eta)]+Y_i(N_i-X_i)[f(\eta+\vec{\delta}_{i,A}-\vec{\delta}_{i,D})-f(\eta)]\Big\}
	\end{aligned}
\end{equation}
describes the exchange of active and dormant individuals in the \emph{same} colony. 

From now on, we denote the process associated with the generator $L$ by 
\begin{equation}
	Z = (Z(t))_{t \geq 0}, \qquad Z(t) = (X_i(t),Y_i(t))_{i\in\Z^d},
\end{equation}
with $X_i(t)$ and $Y_i(t)$ representing the number of type-$\heartsuit$ active and dormant individuals at colony $i$ at time $t$, respectively. Since $Z$ is an interacting particle system, in order to show existence and uniqueness of the process, we can in principle follow the method described by Liggett in \cite[Chapter I, Section 3]{L}. However, for Liggett's method to work, a uniform bound on the sizes $(N_i, M_i)_{i\in\Z^d}$ is needed that we want to avoid. Fortunately, if $L$ is a Markov pregenerator (see \cite[Definition 2.1]{L}), then we can construct the process by providing a unique solution to the martingale problem for $L$. The following proposition tells us that $L$ is indeed a Markov pregenerator and thus prepares the ground for proving the well-posedness of the martingale problem for $L$. 

\begin{proposition}{\bf [Pregenerator]}
	\label{prop2}
	The generator $L$ defined in \eqref{eqn8}, acting on functions in $\mathcal{D}$ defined in \eqref{eqnD}, is a Markov pregenerator.
\end{proposition}

\noindent
The existence of solutions to the martingale problem will be shown by using the techniques described in \cite{L}. In order to establish uniqueness of the solution, we will need to exploit the dual process. 

%%%

\subsubsection{Duality}

The dual process is a block-counting process associated to a spatial version of the interacting seed-bank coalescent described in Section \ref{sub1}. We briefly describe the spatial coalescent process in terms of an interacting particle system. At each site $i\in\Z^d$ there are two reservoirs, an \emph{active} reservoir and a \emph{dormant} reservoir, with $N_i\in\N$ and $M_i\in \N$ labeled locations, respectively. Each location in a reservoir can accommodate at most one particle. As before, we refer to the particles in an active and dormant reservoir as \emph{active} particles and \emph{dormant} particles, respectively. The dynamics of the interacting particle system is as follows (see Figure~\ref{fig:2}).

%%%%%%%%%%%%%%%%%%%%%%%%%%%%%%%%%%%%%%%%%%%%%%
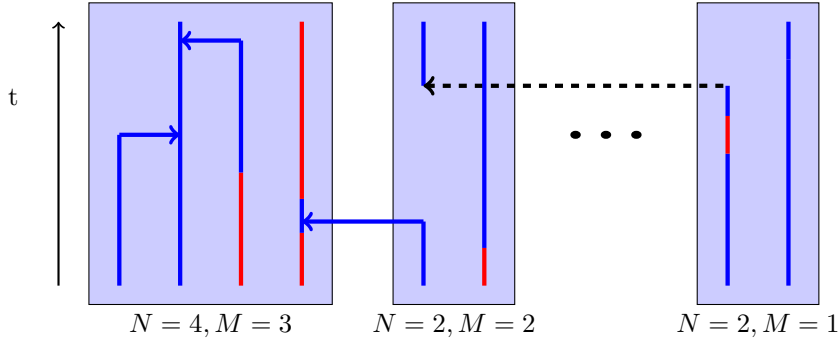
\begin{figure}[htbp]
	\vspace{0.4cm}
	\begin{center}
		\begin{tikzpicture}[xscale=0.8,yscale=0.5]
			\draw[black,thick,->] (0,0) to (0,7);
			\node[left] at (-0.5,5) {t};
			\draw[fill=blue!20!] (0.5,-0.5) rectangle (4.5,7.5);
			\draw[fill=blue!20!] (5.5,-0.5) rectangle (7.5,7.5);
			\node[align=center,below] at (2.5,-0.5) { $N=4,M=3$};
			\node[align=center,below] at (6.5,-0.5) { $N=2,M=2$};
			
			\draw [fill=black, ultra thick] (8.5,4) circle [radius=0.05];
			\draw [fill=black, ultra thick] (9,4) circle [radius=0.05] ;
			\draw [fill=black, ultra thick] (9.5,4) circle [radius=0.05];
			\draw[fill=blue!20!] (10.5,-0.5) rectangle (12.5,7.5);
			\node[align=center,below] at (11.5,-0.5) { $N=2,M=1$};
			
			\draw[ultra thick,blue] (1,0) to (1,4);
			\draw[ultra thick,blue,->] (1,4) to (2,4);
			
			\draw[ultra thick,blue] (2,0) to (2,7);
			
			\draw[ultra thick,red] (3,0) to (3,3);
			\draw[ultra thick,blue] (3,3) to (3,6.5);
			\draw[ultra thick,blue,->] (3,6.5) to (2,6.5);
			
			\draw[ultra thick,red] (4,0) to (4,1.4);
			\draw[ultra thick,blue] (4,1.4) to (4,2.3);
			\draw[ultra thick,red] (4,2.3) to (4,7);
			
			\draw[ultra thick,blue] (6,0) to (6,1.7);
			\draw[ultra thick,blue,->] (6,1.7) to (4,1.7);
			\draw[ultra thick,blue] (6,5.3) to (6,7);
			\draw[ultra thick,black,dashed,<-] (6,5.3) to (11,5.3);
			
			\draw[ultra thick,red] (7,0) to (7,1);
			\draw[ultra thick,blue] (7,1) to (7,7);
			
			\draw[ultra thick,blue] (11,0) to (11,3.5);
			\draw[ultra thick,red] (11,3.5) to (11,4.5);
			\draw[ultra thick,blue] (11,4.5) to (11,5.3);
			
			\draw[ultra thick,blue] (12,0) to (12,6);
			\draw[ultra thick,blue] (12,6) to (12,7);
		\end{tikzpicture}
	\end{center}
	\caption{\small Scheme of transitions in the interacting particle system. Each block depicts the reservoirs located at sites of $\Z^d$. The blue lines represent the evolution of active particles, the red lines represent the evolution of dormant particles.}
	\label{fig:2}
\end{figure}
%%%%%%%%%%%%%%%%%%%%%%%%%%%%%%%%%%%%%%%%%%%%%%%%

\begin{itemize}
	\item 
	An active particle at site $i\in\Z^d$ becomes dormant at rate $\lambda$ by moving to a random labeled location (out of $M_i$ many) in the dormant reservoir at site $i$ when the chosen labeled location is empty, otherwise it remains in the active reservoir.
	\item 
	A dormant particle at site $i\in\Z^d$ becomes active at rate $\lambda K_i$ with $K_i=\tfrac{N_i}{M_i}$ by moving to a random labeled location (out of $N_i$ many) in the active reservoir at site $i$ when the chosen labeled location is empty, otherwise it remains in the dormant reservoir.
	\item 
	An active particle at site $i$ chooses a random labeled location (out of $N_j$ many) from the active reservoir at site $j$ at rate $a(i,j)$ and does the following:
	\begin{itemize}
		\item 
		If the chosen location in the active reservoir at site $j$ is empty, then the particle moves to site $j$ and thereby migrates from the active reservoir at site $i$ to the active reservoir at site $j$.
		\item 
		If the chosen location in the active reservoir at site $j$ is occupied by a particle, then it coalesces with that particle.
	\end{itemize}
\end{itemize}

\noindent
Note that an active particle can migrate between different sites in $\Z^d$ and can coalesce with another active particle even when they are at different sites in $\Z^d$. For simplicity, we will impose the same assumptions on the migration kernel $a(\cdot,\cdot)$ as stated in Assumption \ref{assumpt1}. A configuration $(\eta_i)_{i\in\Z^d}$ of the particle system is an element of $\prod_{i\in\Z^d}\{0,1\}^{N_i}\times\{0,1\}^{M_i}$. For $i\in\Z^d$, $\eta_i$ represents the state of the labeled locations in the active and the dormant reservoir at site $i$ (1 means occupied by a particle, 0 means empty). 

Below we give the definition of the block-counting process associated to the spatial coalescent process described above. Although it is an interesting problem to construct the block-counting process starting from a configuration with infinitely many particles, we will restrict ourselves to configurations with \emph{finitely many particles} only, because this makes the state space countable. Thus, the block-counting process is a continuous-time Markov chain on a countable state space and hence, in the definition below, it suffices to specify the possible transitions and their respective rates only.

\begin{definition}{\bf [Dual]}
	\label{def2}
	{\rm The dual process 
		\begin{equation}
			Z^* = (Z^*(t))_{t \geq 0}, \qquad Z^*(t) = (n_i(t),m_i(t))_{i\in\Z^d},
		\end{equation} 
		is a continuous-time Markov chain with state space 
		\begin{equation}
			\mathcal{X}^* := \Big\{(n_i,m_i)_{i\in\Z^d} \in 
			\prod_{i\in\Z^d}[N_i]\times[M_i]\colon\,\sum_{i\in\Z^d}(n_i+m_i)<\infty\Big\}
		\end{equation}
		and with transition rates
		\begin{equation}
			\label{eqn12}
			\begin{aligned}
				&(n_k,m_k)_{k\in\Z^d} \to\\
				&\begin{cases}
					\displaystyle
					(n_k,m_k)_{k\in\Z^d} - \vec{\delta}_{i,A} &\text{ at rate } 
					\tfrac{2a(i,i)}{N_i}\binom{n_i}{2}\mathbf{1}_{\{n_i\geq 2\}}
					+\sum_{\overset{j\in\Z^d,}{j\neq i}}\tfrac{n_ia(i,j)n_j}{N_j} \text{ for } i\in\Z^d,\\[0.2cm]
					(n_k,m_k)_{k\in\Z^d} - \vec{\delta}_{i,A}+\vec{\delta}_{i,D} &\text{ at rate } 
					\tfrac{\lambda n_i(M_i-m_i)}{M_i} \ \qquad \text{ for } i\in\Z^d,\\[0.2cm]
					(n_k,m_k)_{k\in\Z^d} + \vec{\delta}_{i,A}-\vec{\delta}_{i,D} &\text{ at rate } 
					\tfrac{\lambda (N_i-n_i)m_i}{M_i}\ \qquad \,\text{ for } i\in\Z^d,\\[0.2cm]
					(n_k,m_k)_{k\in\Z^d} - \vec{\delta}_{i,A} +\vec{\delta}_{j,A} &\text{ at rate } 
					\tfrac{n_ia(i,j)(N_j-n_j)}{N_j} \quad \text{ for } i\neq j\in\Z^d,
				\end{cases}
			\end{aligned}
		\end{equation}
		where the configurations $\vec{\delta}_{i,A},\vec{\delta}_{i,D}\in\mathcal{X}^*\subset\mathcal{X}$ are as in \eqref{kron_del}, and additions and subtractions of configurations are performed in accordance with \eqref{eq:addition-subtraction}.
	} \hfill $\Box$
\end{definition}

\noindent
Here, $n_i(t)$ and $m_i(t)$ are the number of active and dormant particles at site $i\in\Z^d$ at time $t$. The first transition describes the coalescence of an active particle at site $i$ with other active particles elsewhere. The second and third transition describe the movement of particles between the active and the dormant reservoir at site $i$. The fourth transition describes the migration of an active particle from site $i$ to site $j$. The following lemma tells us that the dual process $Z^*$ is a well-defined and non-explosive (equivalent to uniqueness) Feller process on the countable state space $\mathcal{X}^*$.

\begin{lemma}{\bf [Uniqueness of dual]}
	\label{lem:unique-dual}
	There exists a unique minimal Feller process $(Z^*(t))_{t\geq 0}$ on $\mathcal{X}^*$ with transition rates given in \eqref{eqn12}.
\end{lemma}

Before we proceed we recall the definition of the martingale problem.

\begin{definition}{\bf [Martingale problem]}
	\label{def3}
	{\rm Suppose that $(L,\mathcal{D})$ is a Markov pregenerator, and let $\eta\in\mathcal{X}$. A probability measure $\P_\eta$ (or, equivalently, a process with law $\P_\eta$) on $D([0,\infty),\mathcal{X})$ is said to solve the martingale problem for $L$ with initial point $\eta$ if 
		\begin{itemize}
			\item 
			$\P_\eta[\xi_{(\cdot)}\in D([0,\infty),\mathcal{X}): \xi_0=\eta]=1$.
			\item 
			$(f(\eta_t)-\int_{0}^{t} (Lf)(\eta_s)\,\d s)_{s \geq 0}$ is a martingale relative to $(\P_\eta,(\mathcal{F}_t)_{t\geq 0})$ for all $f\in\mathcal{D}$, where $(\eta_t)_{t\geq 0}$ is the coordinate process on $D([0,\infty),\mathcal{X})$ and $(\mathcal{F}_t)_{t\geq 0}$ is the filtration given by $\mathcal{F}_t:=\sigma(\eta_s\,|\,s\leq t)$ for $t\geq 0.$ \hfill $\Box$
		\end{itemize}
	}
\end{definition}

The following theorem gives the duality relation between the dual process $Z^*$ and any solution to the martingale problem for $(L,\mathcal{D})$. This type of duality is sometimes referred to as martingale duality.

\begin{theorem}{\bf [Duality relation]}
	\label{T.Mul_Dual}
	Let the process $Z$ with law $\P_\eta$ be a solution to the martingale problem for $(L,\mathcal{D})$ starting from initial state $\eta=(X_i,Y_i)_{i\in\Z^d}\in\mathcal{X}$. Let $Z^*$ be the dual process with law $\P^\xi$ starting from initial state $\xi=(n_i,m_i)_{i\in\Z^d}\in\mathcal{X}^*$. For $t\geq 0$, let $\Gamma(t)$ be the random variable defined by
	\begin{equation}
		\label{def4}
		\Gamma(t) := \max\Big\{\|i\|\colon\,i\in\Z^d,\,\,n_i(s)+m_i(s)>0 \text{ for some } 0 \leq s \leq t\Big\}.
	\end{equation}
	Suppose that the sizes $(N_i)_{i\in\Z^d}$ of the active populations are such that, for any $T>0$,
	\begin{equation}
		\label{assumpt2}
		\sum_{i\in\Z^d} N_i\,\P^\xi\big(\Gamma(T)\geq \|i\|\big) < \infty.
	\end{equation}
	Then, for any $t\geq 0$,
	\begin{equation}
		\label{eqn15}
		\E_\eta\left[\prod_{i\in\Z^d}\frac{\binom{X_i(t)}{n_i}}{\binom{N_i}{n_i}}
		\frac{\binom{Y_i(t)}{m_i}}{\binom{M_i}{m_i}}\mathbf{1}_{\{n_i\leq X_i(t),m_i\leq Y_i(t)\}}\right]
		= \E^\xi\left[\prod_{i\in\Z^d}\frac{\binom{X_i}{n_i(t)}}{\binom{N_i}{n_i(t)}}
		\frac{\binom{Y_i}{m_i(t)}}{\binom{M_i}{m_i(t)}}\mathbf{1}_{\{n_i(t)\leq X_i,m_i(t)\leq Y_i\}}\right],
	\end{equation}
	where the expectations are taken with respect to $\P_\eta$ and $\P^\xi$, respectively.
\end{theorem}

\noindent
Note that the duality function is a product over all colonies of the duality function that appeared in the single-colony model. The infinite products are well-defined: all but finitely many factors are 1, because of our assumption that there are only \emph{finitely many particles} in the dual process. Also note that there is no restriction on $(M_i)_{i\in\Z^d}$, the sizes of the dormant populations. This is because dormant individuals do not migrate and therefore do not feel the spatial extent of the system. 

At first glance it may seem that \eqref{assumpt2} places a severe restriction on $(N_i)_{i\in\Z^d}$, the sizes of the active populations. However, this is not the case. The following corollary provides us with a large class of active population sizes for which Theorem~\ref{T.Mul_Dual} is true under mild assumptions on the migration kernel $a(\cdot,\cdot)$.

\begin{corollary}{\bf [Duality criterion]}
	\label{C.Mul_Dual}
	Suppose that Assumption~{\rm \ref{assumpt1}} is in force. Then \eqref{assumpt2}, and consequently the duality relation in \eqref{eqn15}, hold for every $(N_i)_{i\in\Z^d} \in \mathcal{N}$, where
	\begin{itemize}
		\item[(a)] either
		\begin{equation}
			\label{def5}
			\mathcal{N} := \left\{(N_i)_{i\in\Z^d} \in \N^{\Z^d}\colon\, \lim_{\|i\| \to \infty} \frac{1}{\|i\|} \log N_i = 0\right\}
		\end{equation}
		when $\sum_{i\in\Z^d} \e^{\delta \|i\|}a(0,i) <\infty$ for some $\delta>0$,
		\item[(b)] or
		\begin{equation}
			\mathcal{N} := \left\{(N_i)_{i\in\Z^d} \in \N^{\Z^d}\colon\, \sup_{i\in\Z^d\backslash\{0\}}\frac{N_i}{\|i\|^\delta} < \infty\right\}
		\end{equation}
		when $\sum_{i\in\Z^d} \|i\|^\gamma a(0,i) <\infty$ for some $\delta >0 $ and $\gamma>d+\delta$.
	\end{itemize}
\end{corollary}

\noindent
Corollary~\ref{C.Mul_Dual} shows a \emph{trade-off}: the more we restrict the tails of the migration kernel, the less we need to restrict the sizes of the active populations. 

%%%

\subsubsection{Well-posedness}

We use a martingale problem for the generator $L$ defined in \eqref{eqn8}, in the sense of \cite[p.173]{EK}, to construct $Z$. The following proposition gives existence of solutions for any choice of the reservoir sizes. As for the uniqueness of solutions, we will see that a restriction on the sizes of the active populations is required.

\begin{proposition}{\bf [Existence]}
	\label{P.Mul_Exi}
	Let $L$ be the generator defined in \eqref{eqn8} acting on the set of local functions $\mathcal{D}$ defined in \eqref{eqnD}. Then for all $\eta\in \mathcal{X}$ there exists a solution $\P_\eta$ (a probability measure on $D([0,\infty),\mathcal{X})$) to the martingale problem of $(L,\mathcal{D})$ with initial state $\eta$.
\end{proposition}

The following theorem gives the well-posedness of the martingale problem for $(L,\mathcal{D})$ under a restricted class of sizes of the active populations and thus proves the existence of a unique Feller Markov process describing our multi-colony model.

\begin{theorem}{\bf [Well-posedness]}
	\label{T.Mul_Well}
	Let $(N_i)_{i\in\Z^d}\in\mathcal{N}$ and $(M_i)_{i\in\Z^d}\in\N^{\Z^d}$, and let $L$ be the generator defined in \eqref{eqn8} acting on the set of local functions $\mathcal{D}$ defined in \eqref{eqnD}. Then the following hold:
	\begin{itemize}
		\item 
		For all $\eta\in\prod_{i\in\Z^d}[N_i]\times[M_i]$ there exists a unique solution $Z$ in $D([0,\infty),\mathcal{X})$ of the martingale problem for $(L,\mathcal{D})$ with initial state $\eta$.
		\item 
		$Z$ is Feller and strong Markov, and its generator is an extension of $(L,\mathcal{D})$.
	\end{itemize}
\end{theorem}
In view of the above result, from here onwards, we implicitly assume that the restriction on $(N_i)_{i\in\Z^d}$ to $\mathcal{N}$ is always in force.
%%%
\subsubsection{Equilibrium}
Let us set $Z_i(t):=(X_i(t),Y_i(t))$ for $i\in\Z^d$ and denote by $\mu(t)$ the distribution of $Z(t)$. Further, for each $\theta\in[0,1]$  and $i\in\Z^d,$ let $\nu_{\theta}^i$ be the probability measure on $[N_i]\times[M_i]$ defined as
\begin{equation}
	\nu_\theta^i:=\mathrm{Binomial}(N_i,\theta)\otimes \mathrm{Binomial}(M_i,\theta).
\end{equation}
For $\theta\in[0,1]$, let $\nu_\theta$ be the distribution on $\mathcal{X}$ defined by $\displaystyle\nu_\theta:=\bigotimes_{i\in\Z^d}\nu_\theta^i$ and set
\begin{equation}
	\mathcal{J}:=\{\nu_\theta\,|\,\theta\in[0,1]\}.
\end{equation}
Let $D\,:\,\mathcal{X}\times\mathcal{X}^*\to [0,1]$ be the function defined by
\begin{equation}
	\label{duality.func}
	D((X_k,Y_k)_{k\in\Z^d};(n_k,m_k)_{k\in\Z^d})
	:=\prod_{i\in\Z^d}\frac{\binom{X_i}{n_i}}{\binom{N_i}{n_i}}\frac{\binom{Y_i}{m_i}}{\binom{M_i}{m_i}}
	\mathbf{1}_{\{n_i\leq X_i,m_i\leq Y_i\}}.
\end{equation}

\begin{theorem}{\bf [Convergence to equilibrium]}
	\label{T.Equil}
	Suppose that $\mu(0)=\nu_\theta\in\mathcal{J}$ for some $\theta\in[0,1]$. Then there exists a probability measure $\nu$ determined by the parameter $\theta$ such that
	\begin{itemize}
		\item $\lim\limits_{t\to\infty}\mu(t) = \nu$.
		\item $\nu$ is an equilibrium for the process $Z$.
		\item $\E_\nu [D(Z(0);\eta)] = \lim\limits_{t\to\infty}\E^\eta[\theta^{|Z^*(t)|}]$, where $D(\cdot,\cdot)$ is defined in \eqref{duality.func}, the right expectation is taken w.r.t.\ the dual process $Z^*$ started at configuration $\eta=(n_i,m_i)_{i\in\Z^d}\in\mathcal{X}^*$ and $|Z^*(t)|:=\sum_{i\in\Z^d} [n_i(t)+m_i(t)]$ is the total number of dual particles present at time $t$.
	\end{itemize}
\end{theorem}

\begin{corollary}\label{C.Equil}
	Let $\nu$ be the equilibrium measure of $Z$ in Theorem \ref{T.Equil} corresponding to $\theta\in[0,1]$. Then
	\begin{equation}
		\E_\nu\left[\tfrac{X_i(0)}{N_i}\right]=\E_\nu\left[\tfrac{Y_i(0)}{M_i}\right] = \theta.
	\end{equation}
\end{corollary}

%%%

\subsubsection{Clustering criterion}

We next analyse the long-time behaviour of the multi-colony Moran model with seed-banks. Our interest is to capture the nature of the equilibrium. To be precise, we investigate whether coexistence of different types is possible in equilibrium. The measures $\bigotimes_{i\in\Z^d}\delta_{(0,0)}$ and $\bigotimes_{i\in\Z^d}\delta_{(N_i,M_i)}$ are the trivial equilibria where the system concentrates on only one of the two types. When the system converges to an equilibrium that is not a mixture of these two trivial equilibria, we say that \emph{coexistence} happens. For $i\in\Z^d,$ let us denote the frequency of type-$\heartsuit$ active and dormant individuals at colony $i$ at time $t$ by $x_i(t):=\tfrac{X_i(t)}{N_i}$ and $ y_i(t):=\tfrac{Y_i(t)}{M_i}$ respectively.

\begin{definition}{\bf [Clustering and Coexistence]}
	\label{def;dich}
	{\rm The system is said to exhibit \emph{clustering} if the following hold:
		\begin{itemize}
			\item $\lim\limits_{t\to\infty}\P_\eta(x_i(t)\in\{0,1\})=1,\quad \lim\limits_{t\to\infty}\P_\eta(y_i(t)\in\{0,1\})=1$,
			\item $\lim\limits_{t\to\infty}\P_\eta(x_i(t)\neq x_j(t))=0,\quad \lim\limits_{t\to\infty}\P_\eta(y_i(t)\neq y_j(t))=0$,
			\item  $\lim\limits_{t\to\infty}\P_\eta(x_i(t)\neq y_j(t))=0$,
		\end{itemize}
		for all $i,j\in\Z^d$ and any initial configuration $\eta \in\mathcal{X}$. Otherwise, the system is said to exhibit \emph{coexistence}.}
	\hfill$\Box$
\end{definition}

\noindent
The above conditions make sure that if an equilibrium exists, then it is a mixture of the two trivial equilibria. 

The following criterion, which follows from Theorem \ref{C.Mul_Dual}, gives an equivalent condition for clustering.

\begin{theorem}{\bf [Clustering criterion]}\label{T.Dichotomy}
	The system clusters if and only if in the dual process defined in Definition \ref{def2} two particles, starting from any locations in $\Z^d$ and any states (active or dormant), coalesce with probability $1$.
\end{theorem}

\noindent
Note that the system clusters if and only if the genetic variability at time $t$ between any two colonies converges to 0 as $t\to\infty$. From the duality relation in Theorem \ref{T.Mul_Dual} it follows that this quantity is determined by the state of the dual process starting from two particles.

%%%%%%%%% SECTION 4 %%%%%%%%%%%%%%%%%%%%%%%%

\section{Proofs: duality and equilibrium for the single-colony model}
\label{s.duality}

%%%

Section~\ref{ss.repr} contains the proof of Theorem \ref{T.Sing_Dual}, which follows the algebraic approach to duality described in \cite{CGGR,SSV}. Section~\ref{ss.eq} contains the proof of Proposition \ref{prop1} and Corollary \ref{T.Sin_Eq}, which uses the duality in the single-colony model.

%%%

\subsection{Duality and change of representation}
\label{ss.repr}

Before we proceed with the proof of Theorem~\ref{T.Sing_Dual}, and other results related to \emph{stochastic duality}, it is worth stressing the importance of duality theory. Though originally introduced in the context of interacting particle systems, over the last decade duality theory has gained popularity in various fields, ranging from statistical physics and stochastic analysis to population genetics. One reason behind this wide interests is the simplification that duality provides: it often allows one to extract information about a complex stochastic process through a simpler process. To date, in the literature there exist two systematic approaches towards duality, namely, pathwise construction and Lie-algebraic framework. The former of the two approaches is more practical and widespread in the context of mathematical population genetics \cite{GLW,DK96,HA,JK13}, while the latter has been developed more recently and reveals deeper mathematical structures behind duality, and often also provides a larger class of duality functions (see e.g.\ \cite{CGGR}, \cite{FGG}, \cite{G2018}, \cite{SSV} for a general overview and further references). In what follows, we adopt the Lie algebraic framework suggested by Carinci et al. (2015)\cite{CGGR} and prepare the ground for this setting. The downside is that this approach does not capture the underlying genealogy of the original process. However, it does offer the opportunity to obtain a larger class of duality functions by applying symmetries from the Lie algebra to an already existing duality function \cite{GKRV}. In this paper we refrain from exploring the latter aspect of the Lie-algebraic framework.

We start with briefly recalling that a (real) Lie algebra $\mathfrak{g}$ is a linear space over $\R$ endowed with a so-called \emph{Lie bracket} $[\cdot,\cdot]\colon\,\mathfrak{g}\times\mathfrak{g}\to\mathfrak{g}$ that is bilinear, skew-symmetric and satisfies the \emph{Jacobi identity} \cite{SSV}. The requirement of the bilinearity and skew-symmetry uniquely characterizes a Lie bracket by its action on a basis of $\mathfrak{g}$. An example of a (real) Lie algebra is the well-known $\mathfrak{su}(2)$-algebra, which is the 3-dimensional vector space over $\R$ defined by the action of a Lie bracket on its basis elements $\{J^+, J^-, J^0\}$ as 
\begin{equation}
	\label{eq: su(2) definition}
	[J^0,J^+]=J^+,\quad[J^0,J^-]=-J^-,\quad [J^-,J^+]=-2J^0.
\end{equation} 
For $\alpha\in\N$, let $V_\alpha$ be the linear space of all functions $ f\colon\,[\alpha] \to \R$, and let $\mathfrak{gl}(V_\alpha)$ denote the space of all linear operators on $V_\alpha$. Note that $\mathfrak{gl}(V_\alpha)$ is a  $(1+\alpha)^2$-dimensional Lie algebra with the natural choice of Lie bracket given by $[A,B]:=AB-BA$ for $A,B\in\mathfrak{gl}(V_\alpha)$. Let us define the operators $J^{\alpha,\pm},J^{\alpha,0},A^{\alpha,\pm},A^{\alpha,0}\in\mathfrak{gl}(V_\alpha)$ acting on $ f\,\colon\,[\alpha] \to \R$ as
\begin{align}
	\label{p_def2}
	&J^{\alpha,+}f(n) = (\alpha-n)f(n+1),\quad J^{\alpha,-}f(n) = nf(n-1),\quad J^{\alpha,0}f(n) = (n-\tfrac{\alpha}{2})f(n),\nonumber\\
	&A^{\alpha,+}=J^{\alpha,-}-J^{\alpha,+}-2J^{\alpha,0}, 
	\quad A^{\alpha,-}=J^{\alpha,+},
	\quad A^{\alpha,0}=J^{\alpha,+}+J^{\alpha,0}.
\end{align} 
It is straightforward to see that
\begin{equation}
	[A^{\alpha,0},A^{\alpha,\pm}]=\pm A^{\alpha,\pm},\quad [A^{\alpha,-},A^{\alpha,+}]=-2A^{\alpha,0},
\end{equation}
which are the same commutation relations as in \eqref{eq: su(2) definition}. Thus, for each $\alpha\in\N$, the Lie homomorphism $\phi_\alpha\colon\,\mathfrak{su}(2)\to\mathfrak{gl}(V_\alpha)$ defined by its action on the generators $\{J^+, J^-, J^0\}$ given by 
\begin{equation}
	J^+\mapsto A^{\alpha,+},\quad J^-\mapsto A^{\alpha,-},\quad  J^0\mapsto A^{\alpha,0},
\end{equation}
is a finite-dimensional representation of $\mathfrak{su}(2)$. Similarly, we can verify that $\{J^{\alpha,+},J^{\alpha,-},J^{\alpha,0}\}$, $\alpha \in \N$, form a representation of the \emph{dual} $\mathfrak{su}(2)$-algebra (defined by the commutation relations in \eqref{eq: su(2) definition}, but with opposite signs). 

Below we introduce the notion of duality between two operators and prove a lemma that will be crucial in the proof of duality of both the single-colony and the multi-colony model. The relevance to our context of the above discussion on $\mathfrak{su}(2)$ and its dual algebra will become clear as we go along.

\begin{definition}{\bf [Operator duality]}
	\label{p_def1}
	{\rm Let $A$ and $B$ be two operators acting on functions $f\colon\,\Omega\to\R$ and $g\colon\,\hat{\Omega}\to\R$ respectively. We say that $A$ is dual to $B$ with respect to the duality function $D\colon\,\Omega\times\hat{\Omega}\to\R$, denoted by $A\overset{D}{\longrightarrow}B$, if $(AD(\cdot,y))(x)=(BD(x,\cdot))(y)$ for all $(x,y)\in\Omega\times\hat{\Omega}$.}
	\hfill$\Box$
\end{definition}

\noindent
The following lemma intertwines the $\mathfrak{su}(2)$ and its dual algebra with a duality function.

\begin{lemma}{\bf [Single-colony intertwiner]}
	\label{lem1}
	For $\alpha\in\N$, let $d_\alpha\colon\,[\alpha]\times[\alpha]\to[0,1]$ be the function defined by 
	\begin{equation}
		d_\alpha(x,n) = \frac{\binom{x}{n}}{\binom{\alpha}{n}} \mathbf{1}_{\{n\leq x\}}.
	\end{equation}
	Then the following duality relations hold:
	\begin{equation}
		\label{eqndual}
		J^{\alpha,+}\overset{d_\alpha}{\longrightarrow}A^{\alpha,+}, \quad
		J^{\alpha,-}\overset{d_\alpha}{\longrightarrow}A^{\alpha,-}, \quad
		J^{\alpha,0}\overset{d_\alpha}{\longrightarrow}A^{\alpha,0}.
	\end{equation}
\end{lemma}

\begin{proof}[Proof.]
	By straightforward calculations, it can be shown that $d_\alpha(x,n)$ satisfies the relations
	\begin{equation}
		\begin{aligned}
			(\alpha-x)\,d_\alpha(x+1,n)&=n\,[d_\alpha(x,n-1)-d_\alpha(x,n)]+(\alpha-n)\,[d_\alpha(x,n)-d_\alpha(x,n+1)],\\
			x\,d_\alpha(x-1,n)&=(\alpha-n)\,d_\alpha(x,n),\\
			x\,d_\alpha(x,n)&=(\alpha-n)\,d_\alpha(x,n+1)+n\,d_\alpha(x,n),
		\end{aligned}	
	\end{equation}
	from which the above dualities in \eqref{eqndual} follow immediately.
\end{proof}

\begin{remark}{\bf [Seed-bank and $\mathfrak{su}(2)$-algebra]}
	\label{rem:change-of-representation}
	{\rm The basic idea behind the algebraic approach to duality is to write the generator of a given process in terms of simple operators that form a representation of some known Lie algebra and to make an Ansatz to obtain an intertwiner of the chosen representation. The intertwiner $d_\alpha$ in the above lemma was first identified in \cite[Lemma 1]{G78} as a duality function in disguise for the classical duality between the Moran model and the block-counting process of Kingman's coalescent. Recently, in \cite{CGGR} this duality was put in the algebraic framework by deriving it from an intertwining via $d_\alpha$ of two representations of the \emph{Heisenberg algebra} $\mathscr{H}(2)$. The connection of $d_\alpha$ to the $\mathfrak{su}(2)$-algebra was also made in \cite[Section 3.2]{GKRV}, where the authors obtained a self-duality function of $2j$-SEP factorized in terms of $d_\alpha$ by considering symmetries related to the $\mathfrak{su}(2)$-algebra. The relation of our seed-bank model to the $\mathfrak{su}(2)$-algebra becomes clear once we realize that the seed-bank component in our single-colony model is an \emph{inhomogeneous} version of the $2j$-SEP on two-sites. Thus, it is natural to expect that the classical duality of Moran model can be retrieved from representations of $\mathfrak{su}(2)$-algebra as well. The above lemma indeed provides the ingredients to establish the duality of our single-colony model from representations of the $\mathfrak{su}(2)$-algebra. Although it is possible to guess the dual process of the single-colony model without going into the Lie-algebraic framework, the true usefulness of this approach lies in identifying the dual of the spatial model, where such speculation is no longer feasible.}
	\hfill$\Box$
\end{remark}

\begin{proof}[Proof of Theorem \ref{T.Sing_Dual}.]
	Recall that both $Z=(X(t),Y(t))_{t\geq 0}$ and $Z^*=(n_t,m_t)_{t\geq 0}$ live on the state space $\Omega=[N]\times[M]$. Let $D\colon\,\Omega\times\Omega\to[0,1]$ be the function defined by 
	\begin{equation}
		D\big((X,Y);(n,m)\big) = \frac{\binom{X}{n}}{\binom{N}{n}}\frac{\binom{Y}{m}}{\binom{M}{m}}
		\mathbf{1}_{\{n\leq X,m\leq Y\}}=d_N(X,n)d_M(Y,m), \qquad (X,Y),(n,m)\in\Omega.
	\end{equation}
	Let $G=G_{\text{Mor}}+G_{\text{Exc}}$ be the generator of the process $Z$, where $G_{\text{Mor}},G_{\text{Exc}}$ are as in \eqref{eqn2}--\eqref{eqn3}. Also note from \eqref{def2} that the generator $\widehat{G}$ of the dual process is given by $\widehat{G}=G_{\text{King}}+G_{\text{Exc}}$ where $G_{\text{King}}\colon\,C(\Omega)\to C(\Omega)$ is defined as
	\begin{equation}
		(G_{\text{King}}f)(n,m)=\frac{n(n-1)}{2N}[f(n-1,m)-f(n,m)],\quad(n,m)\in\Omega.
	\end{equation}
	Since $\Omega$ is countable, it is enough to show the generator criterion for duality, i.e., 
	\begin{equation}
		\label{eqndual*}
		\big(GD(\,\cdot\,;\,(n,m))\big)(X,Y)=\big(\widehat{G}D((X,Y);\,\cdot\,)\big)(n,m), \quad (X,Y),(n,m)\in\Omega.
	\end{equation}
	In our notation, \eqref{eqndual*} translates into $G\overset{D}{\longrightarrow}\widehat{G}$. It is somewhat tedious to verify \eqref{eqndual*} by direct computation. Rather, we will write down a proof with the help of the elementary operators defined in \eqref{p_def2}. This approach will also reveal the underlying change of representation of the two operators $G,\widehat{G}$ that is embedded in the duality. 
	
	Note that
	\begin{equation}
		\label{p_rep}
		\begin{aligned}
			G_{\text{King}} &=\tfrac{1}{2N}\left[(A^{N,+}_1-A^{N,-}_1+2A^{N,0}_1)A^{N,0}_1 
			+\tfrac{N}{2}(A^{N,+}_1+A^{N,-}_1-N)\right],\\
			G_{\text{Mor}} &=\tfrac{1}{2N}\left[J^{N,0}_1(J^{N,+}_1-J^{N,-}_1+2J^{N,0}_1) 
			+\tfrac{N}{2}(J^{N,+}_1+J^{N,-}_1-N)\right],\\
			G_{\text{Exc}} &=\tfrac{\lambda}{M}\left[J^{N,+}_1J^{M,-}_2 + J^{N,-}_1J^{M,+}_2 
			+ 2J^{N,0}_1J^{M,0}_2 - \tfrac{NM}{2}\right]\\
			&=\tfrac{\lambda}{M}\left[A^{N,+}_1A^{M,-}_2 + A^{N,-}_1A^{M,+}_2 
			+ 2A^{N,0}_1A^{M,0}_2 - \tfrac{NM}{2}\right],
		\end{aligned}
	\end{equation}
	where the subscripts indicate which variable of the associated function the operators act on. For example, $J^{N,+}_1$ and $J^{M,+}_2$ act on the first and second variable, respectively. So, for a function $f\colon\,[N]\times[M]\to\R$, we have $(J^{N,+}_1f)(n,m)=(J^{N,+}f(\,\cdot\,;\,m))(n)$ and $(J^{M,+}_2f)(n,m)=(J^{M,+}f(n\,;\,\cdot\,))(m)$. The equivalent version of Lemma \ref{lem1} holds for these operators with subscript as well, except that the duality function is $D$. In other words, $J^{N,+}_1\overset{D}{\longrightarrow}A^{N,+}_1$, $J^{M,+}_2\overset{D}{\longrightarrow}A^{M,+}_2$, and so on. Using these duality relations and the representations in \eqref{p_rep}, we have $G_{\text{Mor}}\overset{D}{\longrightarrow}G_{\text{King}}$ and $G_{\text{Exc}}\overset{D}{\longrightarrow}G_{\text{Exc}}$, where we use:
	
	\begin{itemize}
		\item 
		Two operators acting on different sites commute with each other.
		\item 
		For some duality function $d$ and operators $A,B,\hat{A},\hat{B}$, if $A\overset{d}{\longrightarrow}\hat{A},B\overset{d}{\longrightarrow}\hat{B}$, then, for any constants $c_1,c_2$, $ AB\overset{d}{\longrightarrow}\hat{B}\hat{A}$ and $c_1A+c_2B\overset{d}{\longrightarrow}c_1\hat{A}+c_2\hat{B}$.
	\end{itemize} 
	Since $G=G_{\text{Mor}}+G_{\text{Exc}}$ and $\widehat{G}=G_{\text{King}}+G_{\text{Exc}}$, we have $G\overset{D}{\longrightarrow} \widehat{G}$, which proves the claim.
\end{proof}

%%%

\subsection{Equilibrium}
\label{ss.eq}

\begin{proof}[Proof of Proposition \ref{prop1}.]
	For $x\in\R$ and $r\in\N$, let $(x)_r$ be the \emph{falling factorial} defined as
	\begin{equation}
		(x)_r = x(x-1)\times \cdots \times (x-r+1),
	\end{equation}
	where we put $(x)_r=1$ when $r=0$. For any $n\in\N_0$, we can write $x^n$ as
	\begin{equation}\label{p_eqn1}
		x^n=\sum_{j=0}^{n} c_{n,j}\, (x)_j,
	\end{equation}
	where the constants $c_{n,j}$ (known as the Stirling numbers of the second kind) are unique and depend only on $n$ and $j\in[n]$. Let $(n,m)\in\Omega=[N]\times[M]$ be such that $(n,m)\neq (0,0)$, and let $(n_t,m_t)_{t\geq 0}$ be the dual process in Definition \ref{def2}. It follows from \eqref{p_eqn1} and Theorem \ref{T.Sing_Dual} that
	\begin{equation}
		\label{p_eqn2}
		\begin{aligned}
			&\lim\limits_{t\to\infty}\E_{(X,Y)}[X(t)^nY(t)^m]\\ 
			&\qquad = \sum_{i=0}^{n}\sum_{j=0}^{m}c_{n,i}c_{m,j}\lim\limits_{t\to\infty}\E_{(X,Y)}[(X(t))_i(Y(t))_j]\\
			&\qquad =\sum_{i=0}^{n}\sum_{j=0}^{m}c_{n,i}c_{m,j}(N)_i(M)_j\lim\limits_{t\to\infty}\E_{(X,Y)}[D((X(t),Y(t));(i,j))]\\
			&\qquad =\sum_{i=0}^{n}\sum_{j=0}^{m}c_{n,i}c_{m,j}(N)_i(M)_j\lim\limits_{t\to\infty}\E^{(i,j)}[D((X,Y);(n_t,m_t))],
		\end{aligned}
	\end{equation}
	where $D\colon\,\Omega\times\Omega\to[0,1]$ is the duality function in Theorem \ref{T.Sing_Dual}, defined by 
	\begin{equation}
		D((X,Y);(n,m))=\tfrac{\binom{X}{n}}{\binom{N}{n}}\tfrac{\binom{Y}{m}}
		{\binom{M}{m}}\mathbf{1}_{\{n\leq X,m\leq Y\}}\equiv \tfrac{(X)_n(Y)_m}{(N)_n(M)_m},
	\end{equation}
	and the expectation in the last line of \eqref{p_eqn2} is with respect to the dual process. Let $T$ be the first time at which there is only one particle left in the dual, i.e., $ T=\inf\{t>0\colon\,n_t+m_t=1\}$. Note that, for any initial state $(i,j)\in\Omega\backslash\{(0,0)\}$, $T<\infty$ with probability 1, and the distribution of $(n_t,m_t)$ converges as $t\to\infty$ to the invariant distribution $\tfrac{N}{N+M}\delta_{(1,0)}+\tfrac{M}{N+M}\delta_{(0,1)}$. So, for any $(i,j)\in\Omega\backslash\{(0,0)\}$,
	\begin{equation}
		\label{ppre}
		\begin{aligned}
			&\lim\limits_{t\to\infty}\E^{(i,j)}[D((X,Y);(n_t,m_t))]\\
			&\qquad =\lim\limits_{t\to\infty}\E^{(i,j)}[D((X,Y);(n_t,m_t)) \mid T\leq t]\,\P^{(i,j)}(T\leq t)\\
			&\qquad\qquad +\lim\limits_{t\to\infty}\underbrace{\E^{(i,j)}[D((X,Y);(n_t,m_t)) \mid T> t]}_{\leq 1}\P^{(i,j)}(T>t)\\
			&\qquad =\lim\limits_{t\to\infty}\left[\tfrac{X}{N}\,\P^{(i,j)}(n_t=1,m_t=0)+\tfrac{Y}{M}\,\P^{(i,j)}(n_t=0,m_t=1)\right]\\
			&\qquad =\frac{X}{N}\frac{N}{N+M}+\frac{Y}{M}\frac{M}{N+M} =\frac{X+Y}{N+M},
		\end{aligned}
	\end{equation} 
	where we use that the second term after the first equality converges to 0 because $T<\infty$ with probability 1. Combining \eqref{ppre} with \eqref{p_eqn2}, we get
	\begin{equation}
		\begin{aligned}
			&\lim_{t\to\infty}\E_{(X,Y)}[X(t)^nY(t)^m]\\ 
			&\qquad =\sum_{i=0}^{n}\sum_{j=0}^{m}c_{n,i}c_{m,j}(N)_i(M)_j\lim_{t\to\infty}\E^{(i,j)}[D((X,Y);(n_t,m_t))]\\
			&\qquad =\frac{X+Y}{N+M}\left(\sum_{i=0}^{n}c_{n,i}(N)_i\right)\left(\sum_{j=0}^{m}c_{m,j}(M)_j\right)+\left(1-\tfrac{X+Y}{N+M}\right)c_{n,0}c_{m,0}\\
			&\qquad =N^nM^m\frac{X+Y}{N+M},
		\end{aligned}
	\end{equation}
	where the last equality follows from \eqref{p_eqn1} and the fact that $c_{n,0}c_{m,0}=0$ when $(n,m)\neq (0,0)$.
\end{proof}

\begin{proof}[Proof of Corollary \ref{T.Sin_Eq}.]
	Note that the distribution of a two-dimensional random vector $(Z_1,Z_2)$ taking values in $[N]\times[M]$ is determined by the mixed moments $\E[Z_1^iZ_2^j]$, $i,j\in[N]\times[M]$. For $i\in I=[NM]$, let $p_i=\P((Z_1,Z_2)=f^{-1}(i))$, where $f\colon\,[N]\times[M]\to I$ is a bijection. For $i\in I$, let $c_i=\E[Z_1^xZ_2^y]$, where $(x,y)=f^{-1}(i)$. We can write $\vec{c}=A\vec{p}$, where $\vec{p}=(p_i)_{i\in I}, \vec{c}=(c_i)_{i\in I}$ and $A$ is an invertible $(N+1)(M+1)\times(N+1)(M+1)$ matrix. Hence, $\vec{p}=A^{-1}\vec{c}$ is uniquely determined by the mixed moments, and convergence of the mixed moments of $(X(t),Y(t))$ as shown in Proposition \ref{prop1} is enough to conclude that $(X(t),Y(t))$ converges in distribution as $t\to\infty$ to a random vector $(X_\infty,Y_\infty)$ taking values in $[N]\times[M]$. The distribution of $(X_\infty,Y_\infty)$ is also uniquely determined, and is given by $\tfrac{X+Y}{N+M}\delta_{(N,M)}+(1-\tfrac{X+Y}{N+M})\delta_{(0,0)}$.
\end{proof}

%%%%%% SECTION 5 %%%%%%%%%%%%%%%%%%%%%%%%%%%%%%%%%%

\section{Proofs: duality and well-posedness for the multi-colony model}
\label{s.wellposedness}

In Section~\ref{start}, we give the proof of Lemma~\ref{lem:unique-dual}. In Section~\ref{ss.duality}, we introduce equivalent versions for the multi-colony setting of the operators defined in \eqref{p_def2} for the single-colony setting, and use these to prove Theorem \ref{T.Mul_Dual} and Corollary \ref{C.Mul_Dual}. In Section~\ref{ss.wp} we prove Proposition \ref{prop2}, Proposition \ref{P.Mul_Exi} and Theorem \ref{T.Mul_Well}.

%%%

\subsection{Proof of Lemma~\ref{lem:unique-dual}}
\label{start}

\begin{proof}	
	Note that the rate-matrix is nothing but the dual generator $L_\text{dual}$ obtained from the rates specified in \eqref{eqn12}. The action of $L_\text{dual}$ on a function $f\colon\,\mathcal{X}^*\to\R$ is given by
	\begin{equation}
		\label{eq:dual_generator}
		\begin{aligned}
			(L_\text{dual}f)(\xi)
			&=\sum_{i\in\Z^d}\left[\tfrac{n_i(n_i-1)}{2N_i}+n_i\sum_{\overset{j\in\Z^d,}{j\neq i}}a(i,j)
			\tfrac{n_j}{N_j}\right]\big[f(\xi-\vec{\delta}_{i,A})-f(\xi)\big]\\
			&+\sum_{i\in\Z^d}\lambda\,n_i\tfrac{(M_i-m_i)}{M_i}
			\big[f(\xi-\vec{\delta}_{i,A}+\vec{\delta}_{i,D})-f(\xi)\big]\\
			&+\sum_{i\in\Z^d}\lambda(N_i-n_i)\tfrac{\,m_i}{M_i}
			\big[f(\xi+\vec{\delta}_{i,A}-\vec{\delta}_{i,D})-f(\xi)\big]\\
			&+\sum_{i\in\Z^d}\sum_{\overset{j\in\Z^d}{j\neq i}}a(i,j)n_i\tfrac{N_j-n_j}{N_j}
			\big[f(\xi-\vec{\delta}_{i,A}+\vec{\delta}_{j,A})-f(\xi)\big],
		\end{aligned}
	\end{equation}
	where $\xi=(n_i,m_i)_{i\in\Z^d}\in\mathcal{X}^*$ and the configurations $\vec{\delta}_{i,A},\vec{\delta}_{i,D}\in\mathcal{X}^*\subset\mathcal{X}$ are as in \eqref{kron_del}. It is enough to show that $L_\text{dual}$ satisfies the well-known Foster-Lyapunov criterion for stability (see for e.g.\ \cite[Theorem 2.1]{MT93} or \cite[Theorem (1.11)]{Chen91} for Markov processes on countable state spaces), i.e., 
	\begin{equation}
		(L_\text{dual}V)(\xi)\leq p V(\xi)\quad\forall\xi\in\mathcal{X}^*,
	\end{equation}
	for some $p>0$ with $V\colon\,\mathcal{X}^*\to(0,\infty)$ a function such that there exist $(E_k)_{k\in\N}$ with $E_k\uparrow\mathcal{X}^*$ and $\inf_{x\not\in E_k}V(x)\to\infty$ as $k\to\infty$.
	
	Let us define the function $V\,:\,\mathcal{X}^*\to(0,\infty)$ as
	\begin{equation}
		\label{eq:lyapunov-definition}
		V((n_i,m_i)_{i\in\Z^d}):=1+\sum_{i\in\Z^d}(n_i+m_i),\quad(n_i,m_i)_{i\in\Z^d}\in\mathcal{X}^*,
	\end{equation}
	and, for $k\in\N,$ set 
	\begin{equation}
		E_k:=\Big\lbrace(n_i,m_i)_{i\in\Z^d}\in\mathcal{X}^*\,:\,\sum_{i\in\Z^d}n_i+m_i\leq k\Big\rbrace.
	\end{equation} 
	Since $\mathcal{X}^*$ contains configurations with finitely many particles, $V$ is well-defined. It is straightforward to see that
	\begin{equation}
		E_k\uparrow \mathcal{X}^*,\quad\lim\limits_{k\to\infty}\inf_{x\not\in E_k}V(x)=\infty.
	\end{equation}
	Let $\xi=(n_i,m_i)_{i\in\Z^d}\in\mathcal{X}^*$ be arbitrary. Note that, for any $i,j\in\Z^d$ with $i\neq j$,
	\begin{equation}
		\begin{aligned}
			&[V(\xi-\vec{\delta}_{i,A})-V(\xi)] = -\mathbf{1}_{\{n_i\geq 1\}},\\
			&[V(\xi+\vec{\delta}_{i,A}-\vec{\delta}_{i,D})-V(\xi)](N_i-n_i)m_i=0,\\
		\end{aligned}
		\quad
		\begin{aligned}
			&[V(\xi-\vec{\delta}_{i,A}+\vec{\delta}_{i,D})-V(\xi)]\,n_i(M_i-m_i)=0,\\
			&[V(\xi-\vec{\delta}_{i,A}+\vec{\delta}_{j,A})-V(\xi)]\,n_i(N_j-n_j)= 0
		\end{aligned}
	\end{equation}
	and so by using \eqref{eq:dual_generator} we obtain
	\begin{equation}
		\begin{aligned}
			|(L_\text{dual}V)(\xi)|&\leq\sum_{i\in\Z^d}\left[\tfrac{n_i(n_i-1)}{2N_i}+n_i\sum_{\overset{j\in\Z^d,}{j\neq i}}a(i,j)
			\tfrac{n_j}{N_j}\right]|V(\xi-\vec{\delta}_{i,A})-V(\xi)|\\
			&\leq\sum_{i\in\Z^d}\left[\tfrac{n_i}{2}+n_i\sum_{j\in\Z^d}a(i,j)\right]
			\leq \max\{1,c\}\sum_{i\in\Z^d}n_i \leq \max\{1,c\}V(\xi),
		\end{aligned}
	\end{equation}
	where $c=\sum_{i\in\Z^d}a(0,i)<\infty$. Hence, setting $p:=\max\{1,c\}>0$, we have that 
	\begin{equation}
		(L_\text{dual}V)(\xi)\leq |(L_\text{dual}V)(\xi)|\leq p\,V(\xi),
	\end{equation}
	which proves our the claim.
\end{proof}

%%%
\subsection{Duality}
\label{ss.duality}

%%%

\subsubsection{Generators and intertwiners}

Let $f\in C(\mathcal{X})$ and $\eta=(X_i,Y_i)_{i\in\Z^d}\in\mathcal{X}$, and let $\vec{\delta}_{i,A},\vec{\delta}_{i,D}$ be as in \eqref{kron_del}. Define the action of the multi-colony operators as in Table~\ref{tab:multi_oper}. 

%%%%%%%%%%%%%%%%%%%%%%%%%%%%%%%%%%%%%%%%%%%%
\begin{table}[htbp]
	\renewcommand{\arraystretch}{1.5}
	\begin{tabular}{| c | c |}
		\hline
		Operators acting on variable $X_i,i\in\Z^d$ & Operators acting on variable $Y_i,i\in\Z^d$ \\
		\hline
		$ J^{N_i,+}_{i,A}f(\eta) = (N_i-X_i)f(\eta+\vec{\delta}_{i,A}) $  & $ J^{M_i,+}_{i,D}f(\eta) = (M_i-Y_i)f(\eta+\vec{\delta}_{i,D})$   \\
		\hline
		$ J^{N_i,-}_{i,A}f(\eta) = X_i f(\eta-\vec{\delta}_{i,A}) $  & $ J^{M_i,-}_{i,D}f(\eta) = Y_i f(\eta-\vec{\delta}_{i,D}) $   \\
		\hline
		$ J^{N_i,0}_{i,A}f(\eta) = (X_i-\tfrac{N_i}{2})f(\eta) $  & $ J^{M_i,0}_{i,D}f(\eta) = (Y_i-\tfrac{M_i}{2})f(\eta) $\\
		\hline
		$A^{N_i,+}_{i,A}=J^{N_i,-}_{i,A}-J^{N_i,+}_{i,A}-2J^{N_i,0}_{i,A}$ & $A^{M_i,+}_{i,D}=J^{M_i,-}_{i,D}-J^{M_i,+}_{i,D}-2J^{M_i,0}_{i,D}$\\
		\hline
		$A^{N_i,-}_{i,A}=J^{N_i,+}_{i,A}$ & $A^{M_i,-}_{i,D}=J^{M_i,+}_{i,D}$\\
		\hline
		$A^{N_i,0}_{i,A}=J^{N_i,+}_{i,A}+J^{N_i,0}_{i,A}$ & $A^{M_i,0}_{i,D}=J^{M_i,+}_{i,D}+J^{M_i,0}_{i,D}$\\
		\hline
	\end{tabular}
	\caption{Action of operators on $f\in C(\mathcal{X})$.}
	\label{tab:multi_oper}
\end{table}
%%%%%%%%%%%%%%%%%%%%%%%%%%%%%%%%%%%%%%%%%%%%%%%%%%%%%%

\noindent
The same duality relations as in Lemma \ref{lem1} hold for these operators as well. The only difference is that the duality function becomes the site-wise product of the duality functions appearing in the single-colony model. 

\begin{lemma}{\bf [Multi-colony intertwiner]}
	\label{lem2}
	Let $D\colon\,\mathcal{X}\times\mathcal{X}^*\to[0,1]$ be the function defined by
	\begin{equation}
		\label{p_def3}
		D((X_k,Y_k)_{k\in\Z^d};(n_k,m_k)_{k\in\Z^d})=\prod_{i\in\Z^d}\frac{\binom{X_i}{n_i}}{\binom{N_i}{n_i}}\frac{\binom{Y_i}{m_i}}{\binom{M_i}{m_i}}\mathbf{1}_{\{n_i\leq X_i,m_i\leq Y_i\}},
	\end{equation} 
	where $(X_k,Y_k)_{k\in\Z^d}\in\mathcal{X}$ and $(n_k,m_k)_{k\in\Z^d}\in\mathcal{X}^*$. Then, for every $i\in\Z^d$ and $s\in\{0,+,-\}$,
	\begin{equation}
		J^{N_i,s}_{i,A}\overset{D}{\longrightarrow}A^{N_i,s}_{i,A},\quad 
		J^{M_i,s}_{i,D}\overset{D}{\longrightarrow}A^{M_i,s}_{i,D}.
	\end{equation}
\end{lemma}

\begin{proof}[Proof.]
	The proof is exactly the same as the proof of Lemma \ref{lem1}.
\end{proof}

\begin{proposition}{\bf [Generator criterion]}
	\label{prop3}
	Let $L$ be the generator defined in \eqref{eqn8}, and $\hat{L}$ the generator of the dual process defined in Definition {\rm \ref{def2}}. Furthermore, let $D\colon\,\mathcal{X}\times\mathcal{X}^*\to[0,1]$ be the function defined in Lemma {\rm \ref{lem2}}. Then $L\overset{D}{\longrightarrow}\hat{L}$.
\end{proposition}

\begin{proof}[Proof.]
	Recall that $L=L_{\text{Mig}}+L_{\text{Res}}+L_{\text{Exc}}$, where $L_{\text{Mig}},L_{\text{Res}},L_{\text{Ex}}$ are defined in \eqref{eqn9}--\eqref{eqn11}. In terms of the operators defined earlier, these have the following representations:
	\begin{equation}
		\begin{aligned}
			\label{pm_rep1}
			L_{\text{Mig}}&=\sum_{i\in\Z^d}\sum_{\overset{j\in\Z^d}{j\neq i}} 
			\frac{a(i,j)}{N_j}\left[\left(J^{N_i,+}_{i,A}-J^{N_i,-}_{i,A}
			+2J^{N_i,0}_{i,A}\right)J^{N_j,0}_{j,A}+\tfrac{N_j}{2}\left(J^{N_i,+}_{i,A}+J^{N_i,-}_{i,A}-N_i\right)\right],\\[-0.05cm]
			L_{\text{Res}}&=\sum_{i\in\Z^d}\frac{1}{2N_i}\left[J^{N_i,0}_{i,A}\left(J^{N_i,+}_{i,A}-J^{N_i,-}_{i,A}
			+2J^{N_i,0}_{i,A}\right) +\tfrac{N_i}{2}\left(J^{N_i,+}_{i,A}+J^{N_i,-}_{i,A}-N_i\right)\right],\\
			L_{\text{Exc}}&=\sum_{i\in\Z^d}\frac{\lambda}{M_i}\left[J^{N_i,+}_{i,A}J^{M_i,-}_{i,D} 
			+ J^{N_i,-}_{i,A}J^{M_i,+}_{i,D} + 2J^{N_i,0}_{i,A}J^{M_i,0}_{i,D} - \tfrac{N_iM_i}{2}\right]\\
			&=\sum_{i\in\Z^d}\frac{\lambda}{M_i}\left[A^{N_i,+}_{i,A}A^{M_i,-}_{i,D} 
			+ A^{N_i,-}_{i,A}A^{M_i,+}_{i,D} + 2A^{N_i,0}_{i,A}A^{M_i,0}_{i,D} - \tfrac{N_iM_i}{2}\right].
		\end{aligned}
	\end{equation}
	Similarly, the generator $\hat{L}$ of the dual process defined in Definition \ref{def2} acting on $f\in C(\mathcal{X}^*)$ is given by $\hat{L}=\hat{L}_{\text{Mig}}+L_{\text{Exc}}+L_\text{King}$,
	where
	\begin{equation}
		\begin{aligned}
			\hat{L}_{\text{Mig}}f(\xi)&=\sum_{i\in\Z^d}\sum_{\overset{j\in\Z^d}{j\neq i}}
			\frac{a(i,j)}{N_j}\Big\{n_i(N_j-n_j)[f(\xi-\vec{\delta}_{i,A}+\vec{\delta}_{j,A})-f(\xi)]+n_in_j[f(\xi-\vec{\delta}_{i,A})-f(\xi)]\Big\},\\
			L_\text{King}f(\xi)&=\sum_{i\in\Z^d}\frac{n_i(n_i-1)}{2N_i}[f(\xi-\vec{\delta}_{i,A})+f(\xi+\vec{\delta}_{i,A})-2f(\xi)],
		\end{aligned}
	\end{equation}
	for $\xi=(n_i,m_i)_{i\in\Z^d}\in\mathcal{X}^*$. The representations of these operators are
	\begin{align}
		\label{pm_rep2}
		\hat{L}_\text{Mig}&=\sum_{i\in\Z^d}\sum_{\overset{j\in\Z^d}{j\neq i}} 
		\frac{a(i,j)}{N_j}\left[A^{N_j,0}_{j,A}\left(A^{N_i,+}_{i,A}-A^{N_i,-}_{i,A}+2A^{N_i,0}_{i,A}\right)
		+\tfrac{N_j}{2}\left(A^{N_i,+}_{i,A}+A^{N_i,-}_{i,A}-N_i\right)\right],\nonumber\\
		L_\text{King}&=\sum_{i\in\Z^d}\frac{1}{2N_i}\left[\left(A^{N_i,+}_{i,A}-A^{N_i,-}_{i,A}
		+2A^{N_i,0}_{i,A}\right)A^{N_i,0}_{i,A} +\tfrac{N_i}{2}\left(A^{N_i,+}_{i,A}+A^{N_i,-}_{i,A}-N_i\right)\right].
	\end{align}
	From Lemma \ref{lem2} and the representations in \eqref{pm_rep1}--\eqref{pm_rep2}, we see that $L_\text{Mig}\overset{D}{\longrightarrow}\hat{L}_\text{Mig}, \,L_\text{Res}\overset{D}{\longrightarrow}L_\text{King}$ and $L_\text{Ex}\overset{D}{\longrightarrow}L_\text{Ex}$, which yields $L\overset{D}{\longrightarrow}\hat{L}$.
\end{proof}

As shown in \cite[Proposition 1.2]{JK2014}, the generator criterion is enough to get the required duality relation of Theorem \ref{T.Mul_Dual} when both $L$ and $\hat{L}$ are Markov generators of Feller processes. Since it is not a priori clear whether $L$ (or its extension) is a Markov generator, we need to use \cite[Theorem 4.11, Corollary 4.13]{EK}.

%%%

\subsubsection{Proof of duality relation}

\begin{proof}[Proof of Theorem \ref{T.Mul_Dual}.]
	We combine \cite[Theorem 4.11 and Corollary 4.13]{EK} and reinterpret these in our context:
	\begin{itemize}
		\item[$\bullet$]
		Let $(\eta_t)_{t\geq 0}$ and $(\xi_t)_{t\geq 0}$ be two independent processes on $E_1$ and $E_2$ that are solutions to the martingale problem for $(L_1,\mathcal{D}_1)$ and $(L_2,\mathcal{D}_2)$ with initial states $x\in E_1$ and $y\in E_2$. Assume that $D\colon\,E_1\times E_2\to\mathbb{R}$ is such that $D(\,\cdot\,;\,\xi)\in\mathcal{D}_1$ for any $\xi\in E_2$ and $D(\eta\,;\,\cdot)\in\mathcal{D}_2$ for any $\eta\in E_1$. Also assume that for each $T>0$ there exists an integrable random variable $U_T$ such that
		\begin{equation}
			\label{p_eqn3}
			\sup_{0\leq s,t\leq T} | D(\eta_t;\xi_s)| \leq U_T,
			\,\,
			\sup_{0\leq s,t\leq T} |(L_1D(\,\cdot\,;\xi_s))(\eta_t)| \leq U_T,
			\,\,
			\sup_{0\leq s,t\leq T} |(L_2D(\eta_t;\,\cdot\,))(\xi_s)| \leq U_T.
		\end{equation}
		If $(L_1D(\,\cdot\,;y))(x)=(L_2D(x\,;\,\cdot\,))(y)$, then $\E_x[D(\eta_t;y)]=\E^y[D(x,\xi_t)]$ for all $t\geq 0$.
	\end{itemize}
	
	To apply the above, pick $E_1=\mathcal{X}$, $E_2=\mathcal{X}^*$, $L_1=L$, $L_2=L_\text{dual}$, $\mathcal{D}_1=\mathcal{D}$, $\mathcal{D}_2=C(\mathcal{X}^*)$, where $L_\text{dual}$ is the generator of the dual process $Z^*$ and set $D$ to be the function defined in Lemma \ref{lem2}. Note that, since $\mathcal{D}$ contains local functions only, $D(\,\cdot\,;\xi)\in\mathcal{D}$ for any $\xi\in\mathcal{X}^*$ and, since $\mathcal{X}^*$ is countable, $D(\eta\,;\,\cdot\,)\in C(\mathcal{X}^*)$ for any $\eta\in\mathcal{X}$. Fix $x=(X_i,Y_i)_{i\in\Z^d}\in\mathcal{X}$ and $y=(n_i,m_i)_{i\in\Z^d}\in\mathcal{X}^*$. Note that, by Proposition \ref{prop3}, $(L_1D(\,\cdot\,;y))(x)=(L_2D(x\,;\,\cdot\,))(y)$. Pick $(\xi_t)_{t\geq 0}$ to be the process $Z^*$ with initial state $y$. Note that $(\xi_t)_{t\geq 0}$ is the unique solution to the martingale problem for $(L_\text{dual},C(\mathcal{X}^*))$ with initial state $y$. Let $(\eta_t)_{t\geq0}$ denote any solution $Z$ to the martingale problem for $(L,\mathcal{D})$ with initial state $x$. Fix $T>0$ and note that, for $0\leq s,t<T$,
	\begin{equation}
		\label{p_eqn4}
		\begin{aligned}
			(L_1D(\,\cdot\,;\xi_s))(\eta_t)=&\sum_{i\in\Z^d}X_i(t)\left[\,\sum_{j\in\Z^d}
			a(i,j)\tfrac{N_j-X_j(t)}{N_j}\right]\big[D(\eta_t-\vec{\delta}_{i,A};\xi_s)-D(\eta_t;\xi_s)\big]\\
			&+\sum_{i\in\Z^d}(N_i-X_i(t))\left[\,\sum_{j\in\Z^d}a(i,j)\tfrac{X_j(t)}{N_j}\right]
			\big[D(\eta_t+\vec{\delta}_{i,A};\xi_s)-D(\eta_t;\xi_s)\big]\\
			&+\sum_{i\in\Z^d}\lambda X_i(t) \tfrac{M_i-Y_i(t)}{M_i}
			\big[D(\eta_t-\vec{\delta}_{i,A}+\vec{\delta}_{i,D};\xi_s)-D(\eta_t;\xi_s)\big]\\
			&+\sum_{i\in\Z^d}\lambda (N_i-X_i(t))\tfrac{Y_i(t)}{M_i} 
			\big[D(\eta_t+\vec{\delta}_{i,A}-\vec{\delta}_{i,D};\xi_s)-D(\eta_t;\xi_s)\big]
		\end{aligned}
	\end{equation}
	and
	\begin{equation}
		\label{p_eqn5}
		\begin{aligned}
			(L_2D(\eta_t\,;\,\cdot\,))(\xi_s)
			&=\sum_{i\in\Z^d}\left[\tfrac{n_i(s)(n_i(s)-1)}{2N_i}+n_i(s)\sum_{\overset{j\in\Z^d,}{j\neq i}}a(i,j)
			\tfrac{n_j(s)}{N_j}\right]\big[D(\eta_t;\xi_s-\vec{\delta}_{i,A})-D(\eta_t;\xi_s)\big]\\
			&+\sum_{i\in\Z^d}\lambda\,n_i(s)\tfrac{M_i-m_i(s)}{M_i}
			\big[D(\eta_t;\xi_s-\vec{\delta}_{i,A}+\vec{\delta}_{i,D})-D(\eta_t;\xi_s)\big]\\
			&+\sum_{i\in\Z^d}\lambda(N_i-n_i(s))\tfrac{\,m_i(s)}{M_i}
			\big[D(\eta_t;\xi_s+\vec{\delta}_{i,A}-\vec{\delta}_{i,D})-D(\eta_t;\xi_s)\big]\\
			&+\sum_{i\in\Z^d}\sum_{\overset{j\in\Z^d}{j\neq i}}a(i,j)n_i(s)\tfrac{N_j-n_j(s)}{N_j}
			\big[D(\eta_t;\xi_s-\vec{\delta}_{i,A}+\vec{\delta}_{j,A})-D(\eta_t;\xi_s)\big].
		\end{aligned}
	\end{equation}
	The random variable $\Gamma(t)$ defined in Theorem \ref{T.Mul_Dual} is stochastically increasing in time $t$, and if we change the configuration $\eta_t$ outside the box $[0,\Gamma(s)]^d\cap\Z^d$, then the value of $D(\eta_t;\xi_s)$ does not change. Consequently, all the summands in \eqref{p_eqn4} for $\|i\|>\Gamma(s), i\in\Z^d$, are 0, and since $\Gamma(s)\leq \Gamma(T)$ we have the estimate
	\begin{equation}
		\label{p_eqn8}
		|(L_1D(\,\cdot\,;\xi_s))(\eta_t)|\leq 2(c+\lambda)\sum_{\overset{i\in\Z^d}{\|i\| 
				\leq \Gamma(s)}} N_i \ \leq \ 2(c+\lambda)\sum_{\overset{i\in\Z^d}{\|i\|\leq \Gamma(T)}} N_i,
	\end{equation}
	where $c=\sum_{i\in\Z^d}a(0,i)$. Now, by Definition \ref{def2}, the process $(\xi_t)_{t\geq 0}$ is the interacting particle system with coalescence in which the total number of particles can only decrease in time, and so $\sum_{i\in\Z^d}(n_i(s)+m_i(s))\leq N$, where $N=\sum_{i\in\Z^d} (n_i+m_i)$. Also, since $s\leq T$, for $i \in \Z^d$ with $\|i\|>\Gamma(T)$ we have $n_i(s)=m_i(s)=0$. Hence, from \eqref{p_eqn5} we get
	\begin{equation}
		\label{p_eqn9}
		|(L_2D(\eta_t\,;\,\cdot\,))(\xi_s)|\leq 2(c+\lambda)N+2\lambda\sum_{\overset{i\in\Z^d}{\|i\|\leq \Gamma(T)}} N_i.
	\end{equation}
	Define the random variable $U_T$ by
	\begin{equation}
		U_T=1+2(c+\lambda)N+2(c+\lambda)\sum_{\overset{i\in\Z^d}{\|i\|\leq \Gamma(T)}}N_i.
	\end{equation}
	Then, combining \eqref{p_eqn8}--\eqref{p_eqn9} with the fact that the function $D$ takes values in $[0,1]$, we see that $U_T$ satisfies all the conditions in \eqref{p_eqn3}, while assumption \eqref{assumpt2} in Theorem \ref{T.Mul_Dual} ensures the integrability of $U_T$. 
\end{proof}

%%%

\subsubsection{Proof of duality criterion}

\begin{proof}[Proof of Corollary \ref{C.Mul_Dual}.]
	Let $\xi=(n_i,m_i)_{i\in\Z^d}\in\mathcal{X}^*$ and $T>0$ be fixed. By Theorem \ref{T.Mul_Dual}, it suffices to show that, for any $(N_i)_{i\in\Z^d}\in\mathcal{N}$,
	\begin{equation}
		\label{eqn_req}
		\sum_{i\in\Z^d}N_i\,\P^\xi(\Gamma(T) \geq \|i\|)<\infty,
	\end{equation}
	where $\P^\xi$ is the law of the dual process $Z^*$ started from initial state $\xi$. Let $n=\sum_{i\in\Z^d}(n_i+m_i)$ be the initial number of particles, and let $N(t)$ be the total number of migration events within the time interval $[0,t]$. We will construct a Poisson process $N^*$ via coupling such that $N(t)\leq N^*(t)$ for all $t\geq 0$ with probability 1. For this purpose, let us consider $n$ independent particles performing a random walk on $\Z^d$ according to the migration kernel $a(\cdot,\cdot)$. For each $k=1,\ldots,n$, let $\xi_k(t)$ and $\xi^*_k(t)$ denote the position of the $k$-th dependent and independent particle at time $t$, respectively. We take $\xi_k(0)=\xi_k^*(0)$ and couple each $k$-th interacting particle with the $k$-th independent particle as below:
	\begin{itemize}
		\item 
		If the independent particle makes a jump from site $\xi_k^*(t)$ to $j^*\in\Z^d$, then the dependent particle jumps from $\xi_k(t)$ to $j=\xi_k(t)+(j^*-\xi_k^*(t))$ with probability $p_k(t)$ given by
		\begin{equation}
			p_k(t)=
			\begin{cases}
				1-\tfrac{n_j(t)}{N_j} &\text{ if the dependent particle is in an \emph{active} and \emph{non-coalesced} state,}\\
				0 &\text{ otherwise,}
			\end{cases}
		\end{equation}
		where $n_j(t)$ is the number of active particles at site $j$.
		\item 
		The dependent particle does the other transitions (waking up, becoming dormant and coalescence) independently of the previous migration events, with the prescribed rates defined in Definition \ref{def2}. 
	\end{itemize}
	
	Note that, since the migration kernel is translation invariant, under the above coupling the effective rate at which a dependent particle migrates from site $i$ to $j$ is $n_i a(i,j) (1-\tfrac{n_j}{N_j})$ when there are $n_i$ and $n_j$ active particles at site $i$ and $j$, respectively. Also, if $N_k(t)$ and $N_k^*(t)$ are the number of migration steps made within the time interval $[0,t]$ by the $k$-th dependent and independent particle, respectively, then under this coupling $N_k(t)\leq N^*_k(t)$ with probability 1. Set $N^*(\cdot)=\sum_{k=1}^{n}N_k^*(\cdot)$. Then, clearly, 
	\begin{equation}
		\label{eqn_coup}
		N(\cdot)=\sum_{k=1}^{n}N_k(\cdot)\leq N^*(\cdot) \text{ with probability } 1.
	\end{equation}
	Also, $N^*$ is a Poisson process with intensity $cn$, since each independent particle migrates at a total rate $c$. 
	
	Let $Y_l,X_l\in\Z^d$ denote the step at the $l$-th migration event in the dependent and independent particle systems, respectively. Note that $(X_l)_{l\in\N}$ are i.i.d.\ with distribution $(a(0,i))_{i\in\Z^d}$. Since, under the above coupling, a dependent particle copies the step of an independent particle with a certain probability (possibly 0), and $\Gamma(0)$ is the minimum length of the box within which all $n$ dependent particles at time 0 are located, we have, for any $t\geq 0$,
	\begin{equation}
		\label{eqn_ineq}
		\Gamma(t)\leq \Gamma(0)+\sum_{l=1}^{N(t)}|Y_l| \leq \Gamma(0)+\sum_{l=1}^{N^*(t)}|X_l|.
	\end{equation}
	Therefore
	\begin{equation}
		\label{eqn_gam_ineq}
		\P^\xi(\Gamma(T)\geq k) \leq \P\big(S_{N^*(T)} \geq k-\Gamma(0)\big) \qquad \forall\, k \geq 0,
	\end{equation}
	where $ S_{N^*(T)}=\sum_{l=1}^{N*(T)}|X_l|$.
	
	To prove part (a), note that $\E[\e^{\delta S_{N^*(T)}}] < \infty$ and so, by Chebyshev's inequality, 
	\begin{equation}
		\P(S_{N^*(T)}\geq x)=\P(\e^{\delta S_{N^*(T)}}\geq \e^{\delta x})\leq \E[\e^{\delta S_{N^*(T)}}]\,\e^{-\delta x}.
	\end{equation}
	Thus, the inequality in \eqref{eqn_gam_ineq} reduces to
	\begin{equation}
		\label{eqn_assymp}
		\P^\xi(\Gamma(T)\geq k) \leq V\e^{-\delta k} \qquad \forall\, k \geq 0,
	\end{equation}
	where 
	\begin{equation}
		V=\E\left[\exp\{\delta\Gamma(0)+\delta S_{N^*(T)}\}\right] < \infty.
	\end{equation}
	For $k\in\N$, let $\alpha_k = \# \{i\in\Z^d\colon\,\|i\|_\infty=k\}$. Then, $\alpha_k= (2k+1)^d-(2k-1)^d \leq 4^d k^{d-1}$. Hence
	\begin{equation}
		\label{eqn_ineq_1}
		\sum_{i\in\Z^d\backslash\{0\}} N_i\, \P^\xi(\Gamma(T)\geq \|i\|) \leq \sum_{k\in\N} c_k \alpha_k\,\P^\xi(\Gamma(T)\geq k) 
		\leq \sum_{k\in\N} c_k 4^d k^{d-1}\,\P^\xi(\Gamma(T)\geq k),
	\end{equation}
	where $c_k=\sup\{N_i\colon\, \|i\|_\infty=k,i\in\Z^d\}$. Since, under the assumption of part (a), $\lim_{k \to \infty}\tfrac{1}{k}\log c_k = 0$, there exists a $K\in\N$ such that $c_k\leq \e^{\delta k/2}$ for all $k\geq K$. Hence, using \eqref{eqn_assymp}, we find that
	\begin{equation}
		\sum_{i\in\Z^d} N_i\, \P^\xi(\Gamma(T)\geq \|i\|) \leq N_{0}+\sum_{k=1}^{K-1} c_k \alpha_k 
		+ 4^d V\sum_{k=K}^{\infty} k^{d-1}\,\e^{-\delta k/2} < \infty,
	\end{equation}
	which settles part (a).
	
	To prove part (b), note that, under the assumption $\sum_{i\in\Z^d}\|i\|^\gamma a(0,i) <\infty$ for some $\gamma>d+\delta$, we have $\E[S_{N^*(T)}^\gamma]<\infty$, and since $S_{N^*(T)}$ is a positive random variable, we get
	\begin{equation}
		\P(S_{N^*(T)}\geq x)\leq \E[S_{N^*(T)}^\gamma]\,x^{-\gamma}.
	\end{equation}
	From \eqref{eqn_gam_ineq} we get
	\begin{equation}
		\P^\xi(\Gamma(T)\geq k) \leq \frac{V}{(k-\Gamma(0))^\gamma} \qquad \forall\, k > \Gamma(0),
	\end{equation}
	where $V=\E[S_{N^*(T)}^\gamma]$. By the assumption of part (b), there exists a $C >0$ such that 
	\begin{equation}
		c_k=\sup\{N_i\colon\, \|i\|_\infty=k,i\in\Z^d\} \leq Ck^\delta
	\end{equation}
	and so using \eqref{eqn_ineq_1}, we obtain
	\begin{equation}
		\sum_{i\in\Z^d} N_i\, \P^\xi(\Gamma(T) \geq \|i\|) \leq N_0+\sum_{k\leq \Gamma(0)} c_k \alpha_k 
		+ 4^d CV\sum_{k>\Gamma(0)} \frac{k^{d+\delta-1}}{(k-\Gamma(0))^\gamma} < \infty,
	\end{equation}
	which settles part (b).
\end{proof}

%%%

\subsection{Well-posedness}
\label{ss.wp}

In this section we prove Proposition \ref{prop2}, Proposition \ref{P.Mul_Exi} and Theorem \ref{T.Mul_Well}.

%%%

\subsubsection{Existence}

Since the state space $\mathcal{X}$ is compact, the theory described in \cite[Chapter I, Section 3]{L} is applicable in our setting without any significant changes. The interacting particle systems in \cite{L} have state space $W^S$, where $W$ is a compact phase space and $S$ is a countable site space. In our setting, the site space is $S=\Z^d$, but the phase space differs at each site, i.e., $[N_i]\times[M_i]$ at site $i \in \Z^d$. The general form of the generator of an interacting particle system in \cite{L} is
\begin{equation}
	\label{p_eqn10}
	(\Omega f)(\eta) = \sum_{T}\int_{W_T} c_T(\eta,\d\xi)[f(\eta^\xi)-f(\eta)],\quad\quad\eta \in \mathcal{X},
\end{equation}
where the sum is taken over all finite subsets $T$ of $S$, and $\eta^\xi$ is the configuration 
\begin{equation}
	\eta^\xi_i = \begin{cases}
		\xi_i \text{ if } i\in T,\\ \eta_i \text{ else.}
	\end{cases}
\end{equation} 
For finite $T\Subset\mathcal{X}$, $c_T(\eta,d\xi)$ is a finite positive measure on $W_T=W^T$. To make the latter compatible with our setting, we define $W_T=\prod_{i\in T}[N_i]\times[M_i]$. The interpretation is that $\eta$ is the current configuration of the system, $c_T(\eta,W_T)$ is the total rate at which a transition occurs involving \emph{all} the coordinates in $T$, and $c_T(\eta,d\xi)/c_T(\eta,W_T)$ is the distribution of the restriction to $T$ of the new configuration after that transition has taken place. Fix $\eta=(X_i,Y_i)_{i\in\Z^d}\in\mathcal{X}$. Comparing \eqref{p_eqn10} with the formal generator $L$ defined in \eqref{eqn8}, we see that the form of $c_T(\cdot,\cdot)$ is as follows:
\begin{itemize}
	\item $c_T(\eta,d\xi)=0$ if $|T|\geq 2$.
	\item For $|T|=1$, let $T=\{i\}$ for some $i\in\Z^d$. Then $c_T(\eta,\cdot)$ is the measure on $[N_i]\times[M_i]$ given by
	\begin{equation}
		\begin{aligned}
			c_T(\eta,\cdot) &= X_i\left[\,\sum_{j\in\Z^d}a(i,j)\tfrac{N_j-X_j}{N_j}\right]
			\delta_{(X_i-1,Y_i)}(\cdot)+(N_i-X_i)\left[\,\sum_{j\in\Z^d}a(i,j)\tfrac{X_j}{N_j}\right]\delta_{(X_i+1,Y_i)}(\cdot)\\
			&\qquad +\lambda X_i\tfrac{M_i-Y_i}{M_i}\,\delta_{(X_i-1,Y_i+1)}(\cdot)+\lambda(N_i-X_i)\tfrac{Y_i}{M_i}\,
			\delta_{(X_i+1,Y_i-1)}(\cdot).
		\end{aligned}
	\end{equation}
	Note that the total mass is
	\begin{equation}
		\begin{aligned}
			c_T(\eta,W_T) &= X_i\left[\sum_{j\in\Z^d}a(i,j)\tfrac{N_j-X_j}{N_j}\right]+(N_i-X_i)
			\left[\sum_{j\in\Z^d}a(i,j)\tfrac{X_j}{N_j}\right]\\
			&\qquad +\lambda X_i\frac{M_i-Y_i}{M_i}+\lambda(N_i-X_i)\tfrac{Y_i}{M_i}.
		\end{aligned}
	\end{equation}
\end{itemize}

\begin{lemma}{\bf [Bound on rates]}
	\label{lem3}
	Let $c=\sum_{i\in\Z^d} a(0,i) <\infty$. For a finite set $T\Subset \Z^d$, let $c_T=\sup_{\eta\in\mathcal{X}}c_T(\eta,W_T)$. Then $c_T\leq (c+\lambda) \mathbf{1}_{\{|T|=1\}}\sup_{i\in T} N_i$ with $c = \sum_{i\in\Z^d} a(0,i)$.
\end{lemma}

\begin{proof}[Proof.]
	Clearly, $c_T=0$ if $|T|\geq 2$. So let $T=\{i\}$ for some $i\in\Z^d$. We see that, for $\eta=(X_k,Y_k)_{k\in\Z^d}$, $c_T(\eta,W_T)\leq c X_i + c(N_i-X_i) + \lambda X_i  +\lambda(N_i-X_i)=(c+\lambda)N_i=(c+\lambda)\sup_{i\in T} N_i$. 
\end{proof}

\begin{proof}[Proof of Proposition \ref{prop2}.]
	By \cite[Proposition 6.1 of Chapter I]{L}, it suffices to show that
	\begin{equation}
		\label{p_eqn11}
		\sum_{T\ni\, i} c_T < \infty \qquad \forall\, i\in S,
	\end{equation}
	where the sum is taken over all finite subsets $T\Subset S$ containing $i\in S$. Since in our case $S=\Z^d$, we let $i\in \Z^d$ be fixed. By Lemma \ref{lem3}, the sum reduces to $c_{\{i\}}$, and clearly $c_{\{i\}}\leq (c+\lambda) N_i < \infty$. 
\end{proof}

\begin{proof}[Proof of Proposition \ref{P.Mul_Exi}.]
	By  \cite[Proposition 6.1 and Theorem 6.7 of Chapter I]{L}, to show existence of solutions to the martingale problem for $(L,\mathcal{D})$, it is enough to prove that \eqref{p_eqn11} is satisfied. But we already showed this in the proof of Proposition \ref{prop2}. 
\end{proof}

%%%

\subsubsection{Uniqueness}

Before we turn to the proof of Theorem \ref{T.Mul_Well}, we state and prove the following proposition, which, along with the duality established in Corollary \ref{C.Mul_Dual}, will play a key role in the proof of the uniqueness of solutions to the martingale problem.

\begin{proposition}{\bf [Separation]}
	\label{prop4}
	Let $D\colon\,\mathcal{X}\times\mathcal{X}^*\to[0,1]$ be the duality function defined in Lemma \ref{lem2}. Define the set of functions $\mathcal{M}=\{D(\,\cdot\,;\,\xi)\colon\,\xi\in\mathcal{X}^*\}$. Then $\mathcal{M}$ is separating on the set of probability measures on $\mathcal{X}$.
\end{proposition}

\begin{proof}[Proof.]
	Let $\P$ be a probability measure on $\mathcal{X}=\prod_{i\in\Z^d}[N_i]\times[M_i]$. It suffices to show that the finite-dimensional distributions of $\P$ are determined by $\{\int_{\mathcal{X}} f\,d\,\P \colon\,f\in\mathcal{M}\}$. Note that it is enough to show the following:
	\begin{itemize}
		\item[$\bullet$]
		Let $X=(X_1,X_2,\ldots,X_n)\in \prod_{i=1}^{n}[N_i]$ be an $n$-dimensional random vector with some distribution $\P_X$ on $\prod_{i=1}^{n}[N_i]$. Then $\P_X$ is determined by the family
		\begin{equation}
			\mathcal{F}=\left\lbrace\E\left[\,\prod_{i=1}^n\tfrac{\binom{X_i}{\alpha_i}}{\binom{N_i}{\alpha_i}}\right]\colon\, 
			(\alpha_i)_{1\leq i\leq n}\in\prod_{i=1}^n[N_i]\right\rbrace.
		\end{equation}
	\end{itemize}
	By \eqref{p_eqn1}, the family $\mathcal{F}$ is equivalent to the family
	\begin{equation}
		\mathcal{F}^*=\left\lbrace\E\left[\,\prod_{i=1}^nX_i^{\alpha_i}\right]\colon\, 
		(\alpha_i)_{1\leq i\leq n}\in\prod_{i=1}^n[N_i]\right\rbrace
	\end{equation}
	containing the mixed moments of $(X_1,\ldots,X_n)$. Since $X$ takes a total of $N=\prod_{i=1}^n (N_i+1)$ many values, we can write the distribution $\P_X$ as the $N$-dimensional vector $\vec{p}=(p_1,\ldots,p_N)$, where $p_i=\P_X(X=f^{-1}(i))$ and $f\colon\,\prod_{i=1}^{n}[N_i]\to \{1,\dots,N\}$ is the bijection defined by
	\begin{equation}
		f(x_1,x_2,\ldots,x_n) = \sum_{i=1}^{n-1}\left(\prod_{j=i+1}^n(N_j+1)\right)x_i + x_n + 1,
		\quad (x_1,\ldots,x_n)\in\prod_{i=1}^n[N_i].
	\end{equation}
	Note that $\mathcal{F}^*$ also contains $N$ elements, and so we can write $\mathcal{F}^*$ as the $N$-dimensional vector $\vec{e}=(e_1,\ldots,e_N)$, where $e_i=\E[\prod_{k=1}^{n}X_k^{\alpha_k}],(\alpha_1,\ldots,\alpha_n)=f^{-1}(i)$. We show that there exists an invertible linear operator that maps $\vec{p}$ to $\vec{e}$. Indeed, for $i=1,\ldots,n$, define the $(N_i+1)\times(N_i+1)$ Vandermonde matrix $A_i$,
	\begin{equation}
		A_i=\begin{bmatrix}
			1 & 1 & 1 & \dots & 1\\
			\alpha_1 & \alpha_2 & \alpha_3 & \dots & \alpha_{N_i+1}\\
			\alpha_1^2 & \alpha_2^2 & \alpha_3^2 & \dots & \alpha_{N_i+1}^2\\
			\vdots & \vdots & \vdots & \ddots &\vdots \\
			\alpha_1^{N_i} & \alpha_2^{N_i} & \alpha_3^{N_i} & \dots & \alpha_{N_i+1}^{N_i}
		\end{bmatrix},
		\qquad (\alpha_1,\alpha_2\ldots,\alpha_{N_i+1})=(0,1,\ldots,N_i).
	\end{equation}
	Being Vandermonde matrices, all $A_i$ are invertible. Finally, define the $N\times N$ matrix $A$ by $A= A_1 \otimes A_2 \otimes \cdots \otimes A_n$, where $\otimes$ denotes the Kronecker product for matrices. Then $A$ is invertible because all $A_i$ are. Also, we can check that $A\vec{p}=\vec{e}$, and hence the distribution of $X$ given by $\vec{p}=A^{-1}\vec{e}$ is uniquely determined by $\vec{e}$, i.e., the family $\mathcal{F}^*$. 
\end{proof}

\begin{proof}[Proof of Theorem \ref{T.Mul_Well}.]
	We use \cite[Proposition 4.7]{EK}, which states the following (reinterpreted in our setting):
	\begin{itemize}
		\item[$\bullet$]
		Let $\mathcal{S}_1$ be compact and $\mathcal{S}_2$ be separable. Let $x\in\mathcal{S}_1,y\in\mathcal{S}_2$ be arbitrary and $D\colon\,\mathcal{S}_1\times\mathcal{S}_2\to\mathbb{R}$ be such that the set $\{D(\,\cdot\,;z)\colon\,z\in\mathcal{S}_2\}$ is separating on the set of probability measures on $\mathcal{S}_1$. Assume that, for any two solutions $(\eta_t)_{t\geq 0}$ and $(\xi_t)_{t\geq 0}$ of the martingale problem for $(L_1,\mathcal{D}_1)$ and $(L_2,\mathcal{D}_2)$ with initial states $x$ and $y$, the duality relation holds: $\E_x[D(\eta_t,y)]=\E^y[D(x,\xi_t)]$ for all $t\geq 0$. If for every $z\in\mathcal{S}_2$ there exists a solution to the martingale problem for $(L_2,\mathcal{D}_2)$ with initial state $z$, then for every $\eta\in\mathcal{S}_1$ uniqueness holds for the martingale problem for $(L_1,\mathcal{D}_1)$ with initial state $\eta$.
	\end{itemize}
	Pick $\mathcal{S}_1=\mathcal{X}$, $\mathcal{S}_2=\mathcal{X}^*$, $(L_1,\mathcal{D}_1)=(L,\mathcal{D})$ and $(L_2,\mathcal{D}_2)=(L_\text{dual},C(\mathcal{X}^*))$, where $L_\text{dual}$ is the generator of the dual process $Z^*$. Note that in our setting the martingale problem for $(L_\text{dual},C(\mathcal{X}^*))$ is already well-posed (the unique solution is the dual process $Z^*$ in Lemma~\ref{lem:unique-dual}). Hence, combining the above observations with Proposition \ref{prop4} and Corollary \ref{C.Mul_Dual}, we get uniqueness of the solutions to the martingale problem for $(L,\mathcal{D})$ for every initial state $\eta\in\mathcal{X}$.
	
	The second claim follows from \cite[Theorem 6.8 of Chapter I]{L}.
\end{proof}

%%%%%%%%%% SECTION 6 %%%%%%%%%%%%%%%%%%%%%%%%%%%%

\section{Proofs: equilibrium and clustering criterion}
\label{s.criterion}

In Section~\ref{ss.ceq} we prove Theorem \ref{T.Equil} and Corollary \ref{C.Equil}. In Section~\ref{ss.gv} we derive expressions for the single-site genetic variability in terms of the dual process. In Section~\ref{dual_one} we use one dual particle to write down expressions for first moments. In Section~\ref{dual_two} we use two dual particles to write down expressions for second moments. In Section~\ref{ss.couppr} we use these expressions to prove Theorem \ref{T.Dichotomy}.

%%%
\subsection{Convergence to equilibrium}
\label{ss.ceq}
\begin{proof}[Proof of Theorem \ref{T.Equil}.]
	Since the state space $\mathcal{X}$ is compact and thus the set of all probability measures on $\mathcal{X}$ is compact as well, by Prokhorov's theorem. It therefore suffices to prove convergence of the finite-dimensional distributions of $Z(t)=(X_i(t),Y_i(t))_{i\in\Z^d}$. Now recall from the proof of Proposition \ref{prop4} that the distribution of an $n$-dimensional random vector $X(t):=(X_1(t),\ldots,X_n(t))$ taking values in $\prod_{l=1}^{n}[N_l]$ is determined by 
	\begin{equation}
		\mathcal{F}_t=\left\lbrace\E\left[\,\prod_{l=1}^n\tfrac{\binom{X_l(t)}{\alpha_l}}{\binom{N_l}{\alpha_l}}\right]\colon\, 
		(\alpha_l)_{1\leq l\leq n}\in\prod_{l=1}^n[N_l]\right\rbrace.
	\end{equation}
	In fact, the distribution of $X(t)$ converges if and only if $\E\left[\,\prod_{l=1}^n\nicefrac{\binom{X_l(t)}{\alpha_l}}{\binom{N_l}{\alpha_l}}\right]$ converges for all $(\alpha_l)_{1\leq l\leq n}\in\prod_{l=1}^n[N_l]$ as $t\to\infty$. Since our duality function is given by
	\begin{equation}
		D((X_k,Y_k)_{k\in\Z^d};(n_k,m_k)_{k\in\Z^d})=\prod_{i\in\Z^d}\frac{\binom{X_i}{n_i}}{\binom{N_i}{n_i}}\frac{\binom{Y_i}{m_i}}{\binom{M_i}{m_i}}\mathbf{1}_{\{n_i\leq X_i,m_i\leq Y_i\}},
	\end{equation}
	it suffices to show that $\lim_{t\to\infty}\E_{\nu_\theta}[D(Z(t);\eta)]$ exists for all $\eta\in\mathcal{X}^*$. 
	Let $\eta\in\mathcal{X}^*$ be fixed. By duality, we have
	\begin{equation}
		\begin{aligned}
			\E_{\nu_\theta}[D(Z(t);\eta)]
			&=\int_{\mathcal{X}}\E_\xi[D(Z(t);\eta)]\,\d\nu_\theta(\xi)\\ 
			&= \int_{\mathcal{X}}\E^\eta[D(\xi;Z^*(t))]\,\d\nu_\theta(\xi) 
			= \E^\eta\left[\int_{\mathcal{X}}D(\xi;Z^*(t))\,\d\nu_\theta(\xi)\right],
		\end{aligned}
	\end{equation}
	where $\E_\xi$ denotes expectation w.r.t the law of $Z(t)$ started at configuration $\xi\in\mathcal{X}$, $Z^*(t)=(n_i(t),m_i(t))_{i\in\Z^d}$ is the dual process started at configuration $\eta$, and $\E^\eta$ denotes expectation w.r.t the law of the dual process. A simple calculation shows that if $V$ is a random variable with distribution $\mathrm{Binomial}(N,p)$, then $\E\left[\binom{V}{n}/\binom{N}{n}\right]=p^n$ for $0\leq n\leq N$. Since $(X_i(0),Y_i(0))_{i\in\Z^d}$ are all independent under $\nu_\theta$ with Binomials as marginal distributions, we have
	\begin{equation}
		\E_{\nu_\theta}[D(Z(t);\eta)]=\E^\eta\left[\prod_{i\in\Z^d}\theta^{n_i(t)}\,\theta^{m_i(t)}\right] = \E^\eta[\theta^{|Z^*(t)|}],
	\end{equation}
	where $|Z^*(t)|:=\sum_{i\in\Z^d}n_i(t)+m_i(t)$ is total number of particles in the dual process at time $t$. Now, since the dual process is coalescing, $|Z^*(t)|$ is decreasing in $t$. Since $\theta\in[0,1]$, we see that $\E_{\nu_\theta}[D(Z(t);\eta)]$ is increasing in $t$. Thus, $\lim_{t\to\infty}\E_{\nu_\theta}[D(Z(t);\eta)]$ exists, which proves the existence of an equilibrium measure $\nu$ such that the distribution of $Z(t)$ weakly converges to $\nu$. Also, by definition, $\E_\nu[D(Z(0);\eta)]=\lim_{t\to\infty}\E_{\nu_\theta}[D(Z(t);\eta)] =\lim_{t\to\infty}\E^{\eta}[\theta^{|Z^*(t)|}]$.
\end{proof}
\begin{proof}[Proof of Corollary \ref{C.Equil}.]
	This follows by choosing $\eta = \vec{\delta}_{i,A}$ and $\eta=\vec{\delta}_{i,D}$ in the last part of Theorem \ref{T.Equil} and noting that $\E^{\eta}[\theta^{|Z^*(t)|}]=\theta$ when $|\eta|=1$.
\end{proof}

%%%
\subsection{Genetic variability}
\label{ss.gv}

For $i,j\in\Z^d$ and $t\geq0$, define
\begin{equation}
	\label{p_eqn12}
	\Delta_{i,j}(t)= \Delta_{(i,A),(j,A)}(t) + \Delta_{(i,A),(j,D)}(t),
\end{equation}
where
\begin{equation}
	\label{p_eqn13}
	\Delta_{(i,A),(j,A)}(t) =
	\begin{cases}
		\tfrac{X_i(t)(N_j-X_j(t))}{N_i N_j} + \tfrac{X_j(t)(N_i-X_i(t))}{N_j N_i} &\text{ if } i\neq j,\\[0.2cm]
		\tfrac{2X_i(t)(N_i-X_i(t))}{N_i(N_i-1)} &\text{ if } i=j \text{ and } N_i\neq 1,\\[0.2cm]
		0 &\text{ otherwise,}
	\end{cases}
\end{equation}
is the genetic variability (also frequently referred to as `sample heterozygosity') at time $t$ between the active populations of colony $i$ and $j$, i.e., the probability that two individuals drawn randomly from the two populations at time $t$ are of different type, and
\begin{equation}
	\label{p_eqn14}
	\Delta_{(i,A),(j,D)}(t) = \tfrac{X_i(t)(M_j-Y_j(t))}{N_i M_j} + \tfrac{(N_i-X_i(t))Y_j(t)}{N_i M_j}
\end{equation}
is the genetic variability at time $t$ between the active population of colony $i$ and the dormant population of colony $j$. Note that the conditions in Definition \ref{def;dich} are equivalent to
\begin{equation}
	\label{cluster_def}
	\lim\limits_{t\to\infty}\E[\Delta_{i,j}(t)]=0 \qquad \forall\,\, i,j\in\Z^d,
\end{equation}
where the expectation is taken conditional on an arbitrary initial condition $(X_i(0),Y_i(0))_{i\in\Z^d}$, which we suppress from the notation.

We use the dual process to compute $\E(\Delta_{(i,A),(j,A)}(t))$ and $\E(\Delta_{(i,A),(j,D)}(t))$, namely,
\begin{equation}
	\label{p_eqn15}
	\E(\Delta_{(i,A),(j,A)}(t))=
	\begin{cases}
		\E\left[\tfrac{X_i(t)}{N_i}\right)+\E\left(\tfrac{X_j(t)}{N_j}\right]
		-2\,\E\left[\tfrac{X_i(t)X_j(t)}{N_iN_j}\right] & \text{ if } i\neq j,\\[0.2cm]
		2\left(\E\left[\tfrac{X_i(t)}{N_i}\right]-\E\left[\tfrac{X_i(t)(X_i(t)-1)}{N_i(N_i-1)}\right]\right) 
		& \text{ otherwise,}
	\end{cases}
\end{equation}
and
\begin{equation}
	\label{p_eqn16}
	\E[\Delta_{(i,A),(j,D)}(t)]= \E\left[\tfrac{X_i(t)}{N_i}\right]+\E\left[\tfrac{Y_j(t)}{M_j}\right]
	-2\,\E\left(\tfrac{X_i(t)Y_j(t)}{N_iM_j}\right).
\end{equation}
Thus, in terms of the duality function $D$ defined in Lemma \ref{lem2},  
\begin{equation}
	\label{p_eqn17}
	\E[\Delta_{(i,A),(j,A)}(t)] = \E\Big[D(Z(t);\vec{\delta}_{i,A})\Big]+\E\Big[D(Z(t);\vec{\delta}_{j,A})\Big]
	-2\,\E\Big[D(Z(t);\vec{\delta}_{i,A}+\vec{\delta}_{j,A})\Big],
\end{equation} 
where $\vec{\delta}_{i,A}, \vec{\delta}_{j,A}$ are defined in \eqref{kron_del}. Similarly, 
\begin{equation}
	\label{p_eqn18}
	\E[\Delta_{(i,A),(j,D)}(t))]
	=\E\Big[D(Z(t);\vec{\delta}_{i,A})\Big]+\E\Big[D(Z (t);\vec{\delta}_{j,D})\Big]
	-2\,\E\Big[D(Z(t);\vec{\delta}_{i,A}+\vec{\delta}_{j,D})\Big].
\end{equation}
Since, by the duality relation in \eqref{eqn15}, 
\begin{equation}
	\E\Big[D(Z(t);Z^*(0))\Big] = \E\Big[D(Z(0);Z^*(t))\Big],
\end{equation} 
we have
\begin{equation}
	\begin{aligned}
		&\E^{\vec{\delta}_{i,A}}\Big[D(\eta_0;\xi_t)\Big]
		=\E\Big[\tfrac{X_i(t)}{N_i}\Big],\quad \E^{\vec{\delta}_{i,D}}\Big[D(\eta_0;\xi_t)\Big]=\E\Big[\tfrac{Y_i(t)}{M_i}\Big],\\
		&\E^{\vec{\delta}_{i,A}+\vec{\delta}_{j,A}}\Big[D(\eta_0;\xi_t)\Big]
		= \begin{cases}
			\E\Big[\tfrac{X_i(t)(X_i(t)-1)}{N_i(N_i-1)}\Big] & \text{ if } i=j,\\[0.2cm]
			\E\Big[\tfrac{X_i(t)X_j(t)}{N_iN_j}\Big] & \text{ otherwise,}
		\end{cases}\\
		&\E^{\vec{\delta}_{i,A}+\vec{\delta}_{j,D}}\Big(D(\eta_0;\xi_t)\Big)=\E\left(\tfrac{X_i(t)Y_j(t)}{N_iM_j}\right),
	\end{aligned}
\end{equation}
where $\eta_0=Z^*(0)$ and the expectation in the left-hand  side is taken with respect to the dual process $(\xi_t)_{t\geq 0}=Z^*$ defined in Definition \ref{def2}. Combining the above with \eqref{p_eqn17}--\eqref{p_eqn18}, we get
\begin{equation}
	\label{p_eqn19}
	\begin{aligned}
		\E[\Delta_{(i,A),(j,A)}(t)]
		&=\left(\E^{\vec{\delta}_{i,A}}\Big[D(\eta_0;\xi_t)\Big]-\E^{\vec{\delta}_{i,A}+\vec{\delta}_{j,A}}\Big[D(\eta_0;\xi_t)\Big]\right)\\
		&\qquad+\left(\E^{\vec{\delta}_{j,A}}\Big[D(\eta_0;\xi_t)\Big]-\E^{\vec{\delta}_{i,A}+\vec{\delta}_{j,A}}\Big[D(\eta_0;\xi_t)\Big]\right)
	\end{aligned}
\end{equation}
and
\begin{equation}
	\label{p_eqn20}
	\begin{aligned}
		\E[\Delta_{(i,A),(j,D)}(t)]
		&=\left(\E^{\vec{\delta}_{i,A}}\Big[D(\eta_0;\xi_t)\Big]-\E^{\vec{\delta}_{i,A}+\vec{\delta}_{j,D}}\Big[D(\eta_0;\xi_t)\Big]\right)\\
		&\qquad
		+\left(\E^{\vec{\delta}_{j,D}}\Big[D(\eta_0;\xi_t)\Big]-\E^{\vec{\delta}_{i,A}+\vec{\delta}_{j,D}}\Big[D(\eta_0;\xi_t)\Big]\right).
	\end{aligned}
\end{equation}
In Sections~\ref{dual_one}--\ref{dual_two} we derive expressions for the terms appearing in \eqref{p_eqn19}--\eqref{p_eqn20}.

%%%

\subsection{Dual: single particle}
\label{dual_one}

We saw earlier that, in order to compute the first moment of $X_i(t)$ and $Y_i(t)$, we need to put a single particle at site $i$ in the active and the dormant state as initial configurations, respectively. This motivates us to analyse the dual process when it starts with a single particle. The generator $L_\text{dual}$ of the dual process can be written as
\begin{equation}
	\label{p_eqn21}L_{\text{dual}}=L_{\text{Coal}}+L_{AD}+L_{DA}+L_{\text{Mig}},
\end{equation} 
where
\begin{align}
	&(L_{\text{Coal}}f)(\xi)=\sum_{i\in\Z^d}\frac{n_i(n_i-1)}{2N_i}[f(\xi-\vec{\delta}_{i,A})-f(\xi)]
	+\sum_{i\in\Z^d}\sum_{\overset{j\in\Z^d}{j\neq i}}\frac{a(i,j)}{N_j}n_in_j[f(\xi-\vec{\delta}_{i,A})-f(\xi)],\\
	&(L_{AD}f)(\xi)=\sum_{i\in\Z^d}\frac{\lambda\,n_i(M_i-m_i)}{M_i}[f(\xi-\vec{\delta}_{i,A}+\vec{\delta}_{i,D})-f(\xi)],\\
	&(L_{DA}f)(\xi)=\sum_{i\in\Z^d}\frac{\lambda\,m_i(N_i-n_i)}{M_i}[f(\xi+\vec{\delta}_{i,A}-\vec{\delta}_{i,D})-f(\xi)],\\
	&(L_{\text{Mig}}f)(\xi)=\sum_{i\in\Z^d}\sum_{\overset{j\in\Z^d}{j\neq i}}
	\frac{a(i,j)}{N_j}n_i(N_j-n_j)[f(\xi-\vec{\delta}_{i,A}+\vec{\delta}_{j,A})-f(\xi)],
\end{align}
for $f\in C(\mathcal{X}^*)$ and $\xi=(n_i,m_i)_{i\in\Z^d}\in\mathcal{X}^*$. 

When there is a single particle in the system at time 0, and consequently at any later time, the only parts of the generator that are non-zero are $L_{AD}$, $L_{DA}$ and $L_{\text{Mig}}$. Here, $L_{AD}$ turns an active particle at site $i$ into a dormant particle at site $i$ at rate $\lambda$, $L_{DA}$ turns a dormant particle at site $i$ into an active particle at site $i$ at rate $\lambda K_i$, with $K_i=\tfrac{N_i}{M_i}$, while $L_{\text{Mig}}$ moves an active particle at site $i$ to site $j\neq i$ at rate $a(i,j)$. Let us denote the state of the particle at time $t$ by $\xi(t)\in\Z^d\times\{A,D\}$, where the first coordinate of $\xi(t)$ is the location of the particle and the second coordinate indicates whether the particle is active ($A$) or dormant ($D$). Let $\P^\xi$ be the law of the process $(\xi(t))_{t\geq0}$ with initial state $\xi$. 

\begin{lemma}{\bf [First moments]}
	\label{lem_first}
	\begin{equation}
		\begin{aligned}
			\E\left[\frac{X_i(t)}{N_i}\right] 
			&= \sum_{k\in\Z^d}\frac{X_k(0)}{N_k}\,\P^{(i,A)}(\xi(t)=(k,A))+\frac{Y_k(0)}{M_k}\P^{(i,A)}(\xi(t)=(k,D)),\\
			\E\left[\frac{Y_i(t)}{M_i}\right] 
			&= \sum_{k\in\Z^d}\frac{X_k(0)}{N_k}\,\P^{(i,D)}(\xi(t)=(k,A))+\frac{Y_k(0)}{M_k}\P^{(i,D)}(\xi(t)=(k,D)).
		\end{aligned}
	\end{equation}
\end{lemma}

\begin{proof}[Proof.]
	Recall that, via the duality relation,
	\begin{equation}
		\E\left[\frac{X_i(t)}{N_i}\right] = \E^{\vec{\delta}_{i,A}}\left[\prod_{k\in\Z^d}\frac{\binom{X_k(0)}{n_k(t)}}{\binom{N_k}{n_k(t)}} \frac{\binom{Y_k(0)}{m_k(t)}}{\binom{M_k}{m_k(t)}}\mathbf{1}_{\{n_k(t)\leq X_k(0),m_k(t)\leq Y_k(0)\}}\right],
	\end{equation}
	where the expectation in the right-hand side is taken with respect to the dual process with initial state $\vec{\delta}_{i,A}$ (a single active particle at site $i$), which has law $\P_{(i,A)}$. Since the term inside the expectation is equal to $\frac{X_k(0)}{N_k}$ or $\frac{Y_k(0)}{M_k}$, depending on whether $\xi(t)=(k,A)$ or $\xi(t)=(k,D)$, the claim follows immediately. The same argument holds for $\E[\tfrac{Y_i(t)}{M_i}]$ with initial condition $(i,D)$ in the dual process.
\end{proof}

%%%

\subsection{Dual: two particles}
\label{dual_two}

We need to find expressions for the second moments appearing in \eqref{p_eqn15}--\eqref{p_eqn16} in order to fully specify $\E(\Delta_{(i,A),(j,A)}(t))$ and $\E(\Delta_{(i,A),(j,D)}(t))$. This requires us to analyse the dual process starting from two particles. Unlike for the single-particle system, now all parts of the generator $L_{\text{dual}}$ (see \eqref{p_eqn21}) are non-zero, until the two particles coalesce into a single particle. The two particles repel each other: one particle discourages the other particle to come to the same location. The rates in the two-particle system are:
\begin{itemize}
	\item \textbf{(Migration)} 
	An active particle at site $i$ migrates to site $j$ at rate $a(i,j)$ if there is no active particle at site $j$, otherwise at rate $a(i,j)(1-\tfrac{1}{N_j})$.
	\item $\mathbf{(A\to D)}$ 
	An active particle at site $i$ becomes dormant at site $i$ at rate $\lambda$ if there is no dormant particle at site $i$, otherwise at rate $\lambda(1-\tfrac{1}{M_i})$.
	\item $\mathbf{(D\to A)}$ 
	A dormant particle at site $i$ becomes active at site $i$ at rate $\lambda K_i$ if there is no active particle at site $i$, otherwise at rate $\lambda (K_i-\tfrac{1}{M_i})$.
	\item \textbf{(Coalescence)} 
	An active particle at site $i$ coalesces with another active particle at site $j$ at rate $\frac{1}{N_i}$ when $j=i$, otherwise at rate $\tfrac{a(i,j)}{N_j}$.
\end{itemize}
Note that after coalescence has taken place, there is only one particle left in the system, which evolves as the single-particle system. 

Let $(\xi_1(t),\xi_2(t),c(t))\in\mathbb{S}=\mathcal{S}^*\times\mathcal{S}^*\times\{0,1\}$ be the configuration of the two-particle system at time $t$, where $\mathcal{S}^*=\Z^d\times\{A,D\}$. Here $\xi_1(t)$ and $\xi_2(t)$ represent the location and state of the two particles. The variable $c(t)$ takes value 1 if the two particles have coalesced into a single particle by time $t$, and 0 otherwise. It is necessary to add the extra variable $c(t)$ to the configuration in order to make the process Markovian (the rates depend on whether there are one or two particles in the system). To avoid triviality we assume that $c(0)=0$ with probability 1, i.e., two particles at time 0 are always in a non-coalesced state. We denote the law of the process $(\xi_1(t),\xi_2(t),c(t))_{t\geq 0}$ by $\P^\xi$, where the initial condition is $\xi\in\mathcal{S}^*\times\mathcal{S}^*$. It is to be noted that, since the number of active and dormant particles at a site $i$ at any time are limited by $N_i$ and $M_i$, respectively, the two-particle system is not defined whenever it is started from an initial configuration violating the maximal occupancy of the associated sites. Let $\tau$ be the first time at which the coalescence event has occurred, i.e.,
\begin{equation}\label{def_coal}
	\tau=\inf\{t\geq0 \colon\,c(t)=1\}.
\end{equation}
Note that, conditional on $\tau < t$, $\xi_1(s)=\xi_2(s)$ for all $s\geq t$ with probability 1. Define,
\begin{equation}
	M_{(i,\alpha),(j,\beta)}(t)=
	\begin{cases}
		\tfrac{X_i(t)(X_i(t)-1)}{N_i(N_i-1)} & \text{ if } i=j \text{ and } \alpha=\beta=A,\\[0.2cm]
		\tfrac{X_i(t)X_j(t)}{N_iN_j} & \text{ if } i\neq j \text{ and } \alpha=\beta=A,\\[0.2cm]
		\tfrac{Y_i(t)(Y_i(t)-1)}{M_i(M_i-1)} & \text{ if } i=j \text{ and } \alpha=\beta=D,\\[0.2cm]
		\tfrac{Y_i(t)Y_j(t)}{M_iM_j} & \text{ if } i\neq j \text{ and } \alpha=\beta=D,\\[0.2cm]
		\tfrac{X_i(t)Y_j(t)}{N_iM_j} & \text{ if } \alpha=A \text{ and } \beta=D,\\[0.2cm]
		\tfrac{Y_i(t)X_j(t)}{M_iN_j} & \text{ otherwise,}
	\end{cases}
\end{equation}
where $i,j\in\Z^d$ and $\alpha,\beta\in\{A,D\}$. To avoid ambiguity, we set $M_{(i,\alpha),(j,\beta)}(\cdot)=0$ when $((i,\alpha),(j,\beta))$ is not a valid initial condition for the two-particle system.

\begin{lemma}{\bf [Second moments]} 
	\label{lem_second}
	For every valid initial condition $((i,\alpha),(j,\beta))\in(\Z^d\times\{A,D\})^2$ of the two-particle system,
	\begin{equation}
		\begin{aligned}
			\E\left[M_{(i,\alpha),(j,\beta)}(t)\right] = Q((i,\alpha),(j,\beta),t) 
			&+\sum_{k\in\Z^d}\frac{X_k(0)}{N_k}\,\P^{((i,\alpha),(j,\beta))}\big(\xi_1(t)=(k,A),\tau<t\big)\\
			&+\sum_{k\in\Z^d}\frac{Y_k(0)}{M_k}\,\P^{((i,\alpha),(j,\beta))}\big(\xi_1(t)=(k,D),\tau< t\big),
		\end{aligned}
	\end{equation}
	where 
	\begin{equation}
		\begin{aligned}
			&Q((i,\alpha),(j,\beta),t)\\
			&=\sum_{k\in\Z^d}\frac{X_k(0)(X_k(0)-1)}{N_k(N_k-1)}\,\P^{((i,\alpha),(j,\beta))}(\xi_1(t)=\xi_2(t)=(k,A),\tau\geq t)\\
			&\qquad
			+\sum_{\overset{k,l\in\Z^d}{k\neq l}}\frac{X_k(0)X_l(0)}{N_k N_l}\,\P^{((i,\alpha),(j,\beta))}(\xi_1(t)=(k,A),\xi_2(t)=(l,A),\tau\geq t)\\
			&\qquad 
			+\sum_{k,l\in\Z^d}\frac{X_k(0)Y_l(0)}{N_k M_l}\,\P^{((i,\alpha),(j,\beta))}(\xi_1(t)=(k,A),\xi_2(t)=(l,D),\tau\geq t)\\
			&\qquad
			+\sum_{k\in\Z^d}\frac{Y_k(0)(Y_k(0)-1)}{M_k(M_k-1)}\,\P^{((i,\alpha),(j,\beta))}(\xi_1(t)=\xi_2(t)=(k,D),\tau\geq t)\\
			&\qquad
			+\sum_{\overset{k,l\in\Z^d}{k\neq l}}\frac{Y_k(0)Y_l(0)}{M_k M_l}\,\P^{((i,\alpha),(j,\beta))}(\xi_1(t)=(k,D),\xi_2(t)=(l,D),\tau\geq t).
		\end{aligned}
	\end{equation}
\end{lemma}

\begin{proof}[Proof.]
	Note that $M_{(i,\alpha),(j,\beta)}(t)=D(Z(t);\vec{\delta}_{i,\alpha}+\vec{\delta}_{j,\beta})$, where $D$ is the duality function. So, via the duality relation, we have
	\begin{equation}
		\label{mom}
		\E\left[M_{(i,\alpha),(j,\beta)}(t)\right]
		= \E^{\vec{\delta}_{i,\alpha}+\vec{\delta}_{j,\beta}}\left[\prod_{k\in\Z^d}\frac{\binom{X_k(0)}{n_k(t)}}{\binom{N_k}{n_k(t)}} 
		\frac{\binom{Y_k(0)}{m_k(t)}}{\binom{M_k}{m_k(t)}}\mathbf{1}_{\{n_k(t)\leq X_k(0),m_k(t)\leq Y_k(0)\}}\right],
	\end{equation}
	where the expectation in the right-hand side is taken with respect to the dual process when the initial condition has one particle at site $i$ with state $\alpha$ and one particle at site $j$ with state $\beta$, which has law $\P^{((i,\alpha),(j,\beta))}$. Depending on the configuration of the process at time $t$, the right-hand side of \eqref{mom} equals the desired expression.
\end{proof}

The following lemma provides a nice comparison between the one-particle and two-particle system.

\begin{lemma}{\bf [Correlation inequality]}
	\label{lem_corr}
	Let $(\xi(t))_{t\geq 0}$ and $(\xi_1(t),\xi_2(t),c(t))_{t\geq 0}$ be the processes defined in Section \ref{dual_one} and \ref{dual_two}, respectively, and $\tau$ the first time of coalescence defined in \eqref{def_coal}. Then, for any valid initial condition $((i,\alpha),(j,\beta))\in (\Z^d\times\{A,D\})^2$ of the two-particle system and any $(k,\gamma)\in\Z^d\times\{A,D\}$,
	\begin{equation}
		\P^{(i,\alpha)}(\xi(t)=(k,\gamma))\geq \P^{((i,\alpha),(j,\beta))}(\xi_1(t)=(k,\gamma),\tau<t).
	\end{equation}
\end{lemma}

\begin{proof}[Proof.]
	Let $\alpha=A$ and $i,j,k\in\Z^d$ be fixed. Let $\eta=Z(0)$ be the initial configuration defined as,
	\begin{equation}
		(X_n(0),Y_n(0))=
		\begin{cases}
			(N_k,0) &\text{ if } n=k \text{ and } \gamma=A,\\
			(0,M_k) &\text{ if } n=k \text{ and } \gamma=D,\\
			(0,0) &\text{ otherwise,}
		\end{cases}
		\qquad \forall \,n\in\Z^d.
	\end{equation}
	Combining Lemma \ref{lem_first} and Lemma \ref{lem_second}, we get
	\begin{equation}
		\label{gen_var}
		\begin{aligned}
			&\E_\eta\left[\tfrac{X_i(t)}{N_i} - M_{(i,A),(j,\beta)}(t)\right]\\
			\quad&=\sum_{n\in\Z^d}\frac{X_n(0)}{N_n}\big[\P^{(i,A)}(\xi(t)=(n,A))-\P^{((i,A),(j,\beta))}(\xi_1(t)=(n,A),\tau<t)\big]\\
			&\qquad +\sum_{n\in\Z^d}\frac{Y_n(0)}{M_n}\big[\P^{(i,A)}(\xi(t)=(n,D))-\P^{((i,A),(j,\beta))}(\xi_1(t)=(n,D),\tau< t)\big]\\
			&\qquad - Q((i,A),(j,\beta),t)\\
			&= \big[\P^{(i,A)}(\xi(t)=(k,\gamma))-\P^{((i,A),(j,\beta))}(\xi_1(t)=(k,\gamma),\tau<t)\big]-Q((i,A),(j,\beta),t).
		\end{aligned}
	\end{equation}
	Since $Q((i,A),(j,\beta),t)\geq 0$ and the left-hand quantity is positive, we get
	\begin{equation}
		\P^{(i,A)}(\xi(t)=(k,\gamma))\geq \P^{((i,A),(j,\beta))}(\xi_1(t)=(k,\gamma),\tau<t).
	\end{equation}
	Replacing the left-quantity in \eqref{gen_var} with $\E_\eta[\tfrac{Y_i(t)}{M_i}-M_{(i,D),(j,\beta)}(t)]$ and using the same arguments, we see that the inequality for $\alpha=D$ follows.
\end{proof}
%%%

\subsection{Proof of clustering criterion}
\label{ss.couppr}

\begin{proof}[Proof of Theorem \ref{T.Dichotomy}.]
	``$\Longleftarrow$'' First we show that, if $((i,A),(j,\beta))\in(\Z^d\times\{A,D\})^2$ is a valid initial condition for the two-particle system, then
	\begin{equation}
		\label{p_eqn22}
		\lim\limits_{t\to\infty}\E\left[\frac{X_i(t)}{N_i} - M_{(i,A),(j,\beta)}(t)\right] = 0,
		\quad\lim\limits_{t\to\infty}\E\left[\frac{Y_j(t)}{M_j} - M_{(i,A),(j,\beta)}(t)\right] = 0.
	\end{equation}
	Combining Lemma \ref{lem_first} and Lemma \ref{lem_second}, we have
	\begin{equation}
		\label{p_eqn23}
		\begin{aligned}
			&\E\left[\tfrac{X_i(t)}{N_i} - M_{(i,A),(j,\beta)}(t)\right]\\
			\quad&=\sum_{k\in\Z^d}\frac{X_k(0)}{N_k}\big[\P^{(i,A)}(\xi(t)=(k,A))-\P^{((i,A),(j,\beta))}(\xi_1(t)=(k,A),\tau<t)\big]\\
			&\qquad +\sum_{k\in\Z^d}\frac{Y_k(0)}{M_k}\big[\P^{(i,A)}(\xi(t)=(k,D))-\P^{((i,A),(j,\beta))}(\xi_1(t)=(k,D),\tau< t)\big]\\
			&\qquad - Q((i,A),(j,\beta),t).
		\end{aligned}
	\end{equation}
	Using Lemma \ref{lem_corr} and the fact that $Q((i,A),(j,\beta),t)\geq 0$, we have the following:
	\begin{equation}
		\begin{aligned}
			\E\left[\frac{X_i(t)}{N_i} - M_{(i,A),(j,\alpha)}(t)\right] &\leq \sum_{\overset{S\in\{A,D\}}{k\in\Z^d}}
			\big|\P^{(i,A)}(\xi(t)=(k,S))-\P^{((i,A),(j,\beta))}(\xi_1(t)=(k,S),\tau<t)\big|\\
			&=\sum_{\overset{S\in\{A,D\}}{k\in\Z^d}}
			\big[\P^{(i,A)}(\xi(t)=(k,S))-\P^{((i,A),(j,\beta))}(\xi_1(t)=(k,S),\tau<t)\big]\\
			&=1-\P^{((i,A),(j,\beta))}(\tau<t) =\P^{((i,A),(j,\beta))}(\tau\geq t).
		\end{aligned}
	\end{equation}
	Since, by assumption, $\tau<\infty$ with probability 1 irrespective of the initial configuration of the two-particle system, and since the left-hand quantity is positive, we have $\E\big[\tfrac{X_i(t)}{N_i} - M_{(i,A),(j,\beta)}(t)\big]\to 0$ as $t\to\infty$. By a similar argument the other part of \eqref{p_eqn22} is proved as well.
	
	If $((i,A),(j,A))$ is a valid initial condition for the two-particle system, then by using \eqref{p_eqn19}--\eqref{p_eqn20} and \eqref{p_eqn22}, we have
	\begin{equation}
		\begin{aligned}
			\lim_{t\to\infty}\E\Big(\Delta_{(i,A),(j,A)}(t)\Big)
			=\lim_{t\to\infty}\E\left[\tfrac{X_i(t)}{N_i} - M_{(i,A),(j,A)}(t)\right] 
			+ \lim_{t\to\infty}\E\left[\tfrac{X_j(t)}{N_j} - M_{(j,A),(i,A)}(t)\right] = 0.
		\end{aligned}
	\end{equation}
	If $((i,A),(j,A))$ is not a valid initial condition, then we must have that $i=j$ and $N_i=1$, and so $\Delta_{(i,A),(j,A)}(t)=0$ by definition. Thus, for any $i,j\in\Z^d$,
	\begin{equation}
		\lim_{t\to\infty}\E\Big[\Delta_{(i,A),(j,A)}(t)\Big]=0.
	\end{equation}
	Since $((i,A),(j,D))$ is always a valid initial condition for the two-particle system, we also have
	\begin{equation}
		\lim_{t\to\infty}\E\Big[\Delta_{(i,A),(j,D)}(t)\Big]
		=\lim_{t\to\infty}\E\left[\tfrac{X_i(t)}{N_i} - M_{(i,A),(j,D)}(t)\right] 
		+ \lim_{t\to\infty}\E\left[\tfrac{Y_j(t)}{M_j} - M_{(i,A),(j,D)}(t)\right]= 0,
	\end{equation}
	and hence from \eqref{p_eqn12} we have that, for any $i,j\in\Z^d$, $\E(\Delta_{i,j}(t))\to 0$ as $t\to\infty$, which proves the claim.
	
	\medskip\noindent
	``$\Longrightarrow$'' Suppose that the system clusters for any initial configuration $Z(0)\in\mathcal{X}$. Then, by dominated convergence, the system clusters for any initial distribution of $Z(0)$ as well. Fix $\theta\in(0,1)$, and let the distribution of $Z(0)$ be $\nu_{\theta}$, where
	\begin{equation}
		\begin{aligned}
			\nu_{\theta} = \bigotimes_{i\in\Z^d} (\mathrm{Binomial}(N_i,\theta)\otimes \mathrm{Binomial}(M_i,\theta)).
		\end{aligned}
	\end{equation}
	We will prove via contradiction that in the dual two particles with arbitrary valid initial states coalesce with probability 1, i.e., $\tau < \infty$ with probability 1. Indeed, suppose that this is not true, i.e., for some valid initial configuration $(\xi_1,\xi_2)\in\mathcal{S}^*\times\mathcal{S}^*$ of the two-particle system we have $\P^{(\xi_1,\xi_2)}(\tau=\infty)>0$, where $\mathcal{S}^*=\Z^d\times\{A,D\}$. Since the dual process with two particles is irreducible (any valid configuration is accessible), we have $\P^{\xi}(\tau=\infty)>0$ for any valid initial condition $\xi\in \mathcal{S}^*\times\mathcal{S}^*$. Let $\rho:=\P^{((i,A),(i,D))}(\tau=\infty)>0$, where $i\in\Z^d$ is fixed. Note that $((i,A),(i,D))$ is always a valid initial condition for the two-particle system, since $N_i,M_i\geq 1$. Let $\P^{(i,A)}$ be the law of the single-particle process $(\xi(t))_{t\geq 0}$ started with initial condition $(i,A)$.

	Since, by assumption, $\E_{\nu_\theta}\left[\Delta_{(i,A),(i,D)}(t)\right]\to 0$ as $t\to\infty$,
	we must have
	\begin{equation}
		\label{p_eqn24}
		\lim_{t\to\infty}\E_{\nu_{\theta}}\left[\tfrac{X_i(t)(M_i-Y_i(t))}{N_iM_i}\right]=0.
	\end{equation}
	Since $((i,A),(i,D))$ is a valid initial condition for the two-particle system, by using \eqref{p_eqn23} with $\nu_{\theta}$ as initial distribution we get
	\begin{equation}
		\label{p_eqn25}
		\begin{aligned}
			\E_{\nu_{\theta}}\left[\tfrac{X_i(t)(M_i-Y_i(t))}{N_iM_i}\right]
			&=\E_{\nu_{\theta}}\left[\tfrac{X_i(t)}{N_i} - M_{(i,A),(i,D)}(t)\right]\\
			&=\sum_{n\in\Z^d}\E_{\nu_{\theta}}\left[\frac{X_n(0)}{N_n}\right] \big[\P^{(i,A)}(\xi(t)=(n,A))
			-\P^{((i,A),(i,D))}(\xi_1(t)=(n,A),\tau<t)\big]\\
			&\quad +\sum_{n\in\Z^d}\E_{\nu_{\theta}}\left[\frac{Y_n(0)}{M_n}\right]\big[\P^{(i,A)}(\xi(t)=(n,D))
			-\P^{((i,A),(i,D))}(\xi_1(t)=(n,D),\tau< t)\big]\\
			&\quad - \E_{\nu_{\theta}}\big[Q((i,A),(i,D),t)\big],
		\end{aligned}
	\end{equation}
	where $Q(\cdot,\,\cdot,\,\cdot)$ is defined as in Lemma \eqref{lem_second}. Since, under $\nu_{\theta}$, $(X_n(0))_{n\in\Z^d}$, $(Y_n(0))_{n\in\Z^d}$ are all independent of each other and $X_n(0)$ and $Y_n(0)$ have distributions $\mathrm{Binomial}(N_n,\theta)$ and $\mathrm{Binomial}(M_n,\theta)$, respectively, we have
	\begin{equation}
		\begin{aligned}
			&\E_{\nu_{\theta}}\left[\frac{X_n(0)}{N_n}\right] = \E_{\nu_{\theta}}\left[\frac{Y_n(0)}{M_n}\right] = \theta,\\
			&\E_{\nu_{\theta}}\left[\frac{X_n(0)(X_n(0)-1)}{N_n(N_n-1)}\right] =\theta^2&\text{ if } N_n\neq 1,\\
			&\E_{\nu_{\theta}}\left[\frac{Y_n(0)(Y_n(0)-1)}{M_n(M_n-1)}\right] =\theta^2&\text{ if } M_n\neq 1.
		\end{aligned}
	\end{equation}
	Hence $\E_{\nu_{\theta}}\left[Q((i,A),(i,D),t)\right] = \theta^2\,\P^{((i,A),(i,D))}(\tau\geq t)$, and thus \eqref{p_eqn25} reduces to
	\begin{equation}
		\label{last}
		\begin{aligned}
			\E_{\nu_{\theta}}\left[\tfrac{X_i(t)(M_i-Y_i(t))}{N_iM_i}\right]
			&=\theta\,\left[1-\P^{((i,A),(i,D))}(\tau < t)\right] -\theta^2\,\P^{((i,A),(i,D))}(\tau \geq t)\\
			&=\theta(1-\theta)\,\P^{((i,A),(i,D))}(\tau \geq t).
		\end{aligned}
	\end{equation} 
	By \eqref{p_eqn24}, the left-hand side of \eqref{last} tends to $0$ as $t\to\infty$. Because $\theta \in (0,1)$, we have
	\begin{equation}
		\rho = \lim\limits_{t\to\infty}\P^{((i,A),(i,D))}(\tau \geq t) = 0,
	\end{equation}
	which is a contradiction.
\end{proof}

%%%%%%%%%%%%% REFERENCES %%%%%%%%%%%%%%%%%%%%%%%

\bibliographystyle{spmpsci}
\bibliography{reference-arxiv-final} 

%%%%%%%%%%%%%%%%%%%%%%%%%%%%%%%%%%%%%%%%%%%%%
 
\end{document}